\documentclass[11pt,a4paper,english]{amsart}

\usepackage[T1]{fontenc}
\usepackage[utf8]{inputenc}
\usepackage{amsmath}
\usepackage{amssymb}
\usepackage{amsthm}
\usepackage{latexsym}
\usepackage{amsfonts}

\usepackage[onehalfspacing]{setspace} 
\setdisplayskipstretch{.421} 
\newlength{\abovebis} 
\newlength{\belowbis} 
\newlength{\aboveshortbis} 
\newlength{\belowshortbis} 
\everydisplay\expandafter{%
  \the\everydisplay 
  \advance\abovedisplayskip\abovebis 
  \advance\belowdisplayskip\belowbis 
  \advance\abovedisplayshortskip\aboveshortbis 
  \advance\belowdisplayshortskip\belowshortbis 
}

\allowdisplaybreaks[3]

\def\R{\mathbb{R}}

\def\N{\mathbb{N}}
\def\C{\mathbb{C}}

\def\Ree{\mathrm{Re}}
\def\Imm{\mathrm{Im}}

\def\supp{\mathrm{supp}\,}
\def\E{\mathcal{E}}
\def\SE{\Sigma_E}

\theoremstyle{plain}

\newtheorem{lem}{Lemma}[section]
\newtheorem{theo}[lem]{Theorem}
\newtheorem{cor}[lem]{Corollary}

\newtheorem{prop}[lem]{Proposition}

\theoremstyle{definition}

\newtheorem{rem}{Remark}[section]

\numberwithin{equation}{section}

\begin{document}
\title[H\"older-logarithmic stability estimate in 2D at positive energy]{A H\"older-logarithmic stability estimate for an inverse problem in two dimensions}
\author{Matteo Santacesaria}
\address[M. Santacesaria]{Centre de Mathématiques Appliquées -- UMR 7641, \'Ecole Polytechnique, 91128, Palaiseau, France}
\email{santacesaria@cmap.polytechnique.fr}
\begin{abstract}
The problem of the recovery of a real-valued potential in the two-dimensional Schr\"odinger equation at positive energy from the Dirichlet-to-Neumann map is considered. It is know that this problem is severely ill-posed and the reconstruction of the potential is only logarithmic stable in general. In this paper a new stability estimate is proved, which is explicitly dependent on the regularity of the potentials and on the energy. Its main feature is an efficient \textit{increasing stability} phenomenon at sufficiently high energies: in some sense, the stability rapidly changes from logarithmic type to H\"older type. The paper develops also several estimates for a non-local Riemann-Hilbert problem which could be of independent interest.
\end{abstract}

\maketitle

\section{Introduction}
This paper is the last of a series of four papers focusing on stability estimates for the Gel'fand-Calder\'on problem on the plane. In the first paper, \cite{NS}, a first global stability estimate is proved. The second and the third paper, \cite{S1, S2}, deal with stability estimates in the zero-energy and negative-energy case, respectively, which explicitly depend on the energy and on the regularity of potentials. The present work covers the last and maybe more interesting case when the energy is supposed to be positive.

The Gel'fand-Calder\'on problem concerns the Schr\"odinger equation at fixed energy $E$,
\begin{equation} \label{equa}
(-\Delta + v)\psi = E \psi \quad \text{on } D, \quad E \in \R,
\end{equation}
where $D$ is a open bounded domain in $\R^2$ and $v \in L^{\infty}(D)$ (we will refer to $v$ as a \textit{potential}).
Under the assumption that
\begin{equation} \label{direig}
0 \textrm{ is not a Dirichlet eigenvalue for the operator } - \Delta + v -E\textrm{ in } D,
\end{equation}
we can define the Dirichlet-to-Neumann operator $\Phi(E) :H^{1/2}(\partial D) \to H^{-1/2}(\partial D)$, corresponing to the potential $v$, as follows:
\begin{equation} \label{defdtn1}
\Phi(E)f = \left. \frac{\partial u}{\partial \nu}\right|_{\partial D},
\end{equation}
where $f \in H^{1/2}(\partial D)$, $\nu$ is the outer normal of $\partial D$, and $u$ is the $H^1(D)$-solution of the Dirichlet problem
\begin{equation}
(-\Delta + v)u = Eu \; \textrm{on} \; D, \; \; \; u|_{\partial D}=f.
\label{schr}
\end{equation}
This construction gives rise to the so-called Gel'fand-Calder\'on problem.

\textbf{Problem 1.} Given $\Phi(E)$ for a fixed $E \in \R$, find $v$ on $D$.

This problem can be considered as the Gel’fand inverse boundary value problem for the two-dimensional Schr\"odinger equation at fixed energy (see \cite{G}, \cite{N1}). At zero energy this problem can also be seen as a generalization of the Calder\'on problem of the electrical impedance tomography (see \cite{C}, \cite{N1}).

Note that this problem is not overdetermined, in the sense that we consider the reconstruction of a function $v$ of two variables from inverse problem data dependent on two variables.\smallskip

In this paper we study interior stability estimates, i.e. we want to prove that given two Dirichlet-to-Neumann operators $\Phi_1(E)$ and $\Phi_2(E)$, corresponding to potentials $v_1$ and $v_2$ on $D$, we have that
\begin{equation} \nonumber
\|v_1 - v_2\|_{L^{\infty}(D)} \leq \omega \left( \| \Phi_1(E) - \Phi_2(E)\|_*\right),
\end{equation}
where the function $\omega(t) \to 0$ as fast as possible as $t \to 0$ at any fixed $E$ and $\| \cdot \|_*$ is some operator norm. The explicit dependence of $\omega$ on $E$ is analysed as well.

There is a wide literature on the Gel'fand-Calder\'on inverse problem. In the case of complex-valued potentials the global injectivity of the map $v \to \Phi$ was firstly proved in \cite{N1} for $D \subset \R^d$ with $d \geq 3$ and in \cite{B} for $d = 2$ with $v \in L^p$: in particular, these results were obtained by the use of global reconstructions developed in the same papers. A global logarithmic stability estimate for Problem 1 for $d \geq 3$ was first found by Alessandrini in \cite{A}. In the two-dimensional case the first global stability estimate was given in \cite{NS}. In \cite{IN} logarithmic stability was proved in dimension $d \geq 2$ without condition \eqref{direig}, using more general boundary data (impedance boundary map). For Lipschitz stability estimates concerning this and similar inverse problems with finite dimensional restrictions see \cite{AV}, \cite{BV}, \cite{BDQ} and \cite{Bou}.

In \cite{S1} and in \cite{S2} we considered Problem 1 at zero and negative energy, respectively, and answered the following question: how the stability estimates vary with respect to the smoothness of the potentials and the energy. 

This paper completes the preceding works by considering the positive energy case.

We will assume for simplicity that
\begin{equation} \label{cv1}
\begin{split}
&D \text{ is an open bounded domain in } \R^2, \qquad \partial D \in C^2, \\
&v \in W^{m,1}(\R^2) \text{ for some } m > 2, \quad \bar v = v, \quad \mathrm{supp} \; v \subset D,
\end{split}
\end{equation} 
where
\begin{align}
&W^{m,1}(\R^2) = \{ v \; : \; \partial^J v \in L^1(\R^2),\; |J| \leq m \}, \qquad m \in \N \cup \{0\},\\ \nonumber
&J \in (\N \cup \{0\})^2, \qquad |J| = J_1+J_2, \qquad \partial^J v(x) = \frac{\partial^{|J|} v(x)}{\partial x_1^{J_1} \partial x_2^{J_2}}.
\end{align}
Let
\begin{equation} \nonumber
\|v\|_{m,1} = \max_{|J| \leq m} \| \partial^J v \|_{L^1(\R^2)}.
\end{equation}
We will need the following regularity condition:
\begin{equation} \label{cv2}
E > E_1, 
\end{equation}
where $E_1 = E_1(\|v\|_{m,1},D)$ or, roughly speaking, $E$ is sufficiently great with respect to some appropriate norm of the potential. This condition implies, in particular, that the Faddeev eigenfunctions are well-defined on the entire fixed-energy surface in the spectral parameter (see Section \ref{sec2} and Remark \ref{reme}).

\begin{theo} \label{maintheo}
Let the conditions \eqref{direig}, \eqref{cv1}, \eqref{cv2} hold for the potentials $v_1, v_2$, where $D$ is fixed, and let $\Phi_1(E)$ , $\Phi_2(E)$ be the corresponding Dirichlet-to-Neumann operators at fixed positive energy $E > 0$. Let $\|v_j\|_{m,1} \leq N$, $j=1,2$, for some $N >0$. Then there exists a constant $C_1 = C_1(D, N,m)$ such that for any $0 < \tau \leq  1$, we have
\begin{align} \label{est2}
\|v_2-v_1\|_{L^{\infty}(D)}&\leq C_1\left( {E}\delta^{\tau}+ \left(\sqrt{E}+(1-\tau) \log(3+\delta^{-1})\right)^{-(m-2)}\right),
\end{align}
for every $\delta < \tilde \delta({\tau})$, where $\delta =\|\Phi_2(E)-\Phi_1(E)\|_{L^{\infty}(\partial D)\to L^{\infty}(\partial D)}$.
\end{theo}
This results yields the following corollary.
\begin{cor}
Under the same assumptions, there exists a constant $C_2=C_2(E,D,N,m)$ such that
\begin{align} \label{est1}
\|v_2-v_1\|_{L^{\infty}(D)}&\leq C_2(\log(3+\delta^{-1}))^{-\alpha}, \qquad \alpha = m-2,
\end{align}
for $\delta < \tilde \delta$.
\end{cor}

The novelty of estimate \eqref{est1}, with respect to \cite{NS}, is that, as $m \to +\infty$, we have $\alpha \to +\infty$. Moreover, under the assumption of Theorem \ref{maintheo}, according to instability estimates of Mandache \cite{M} and Isaev \cite{I}, our result is almost sharp. To be more precise, it was proved that stability estimate \eqref{est1} cannot hold for $\alpha > 2m$ for $C^m$ real-valued potentials and $\alpha > m$ for $C^m$ complex-valued potentials. Note that stability estimates and instability counterexamples are proved in different function spaces. In particular, by Sobolev embedding, we have only that $W^{m,1+\varepsilon}(D) \subset C^{m-2}(D)$ for any $\varepsilon > 0$. From this and the fact that the same stability holds in the linearized case (Born approximation, see \cite{NN}), we believe that our result is in fact sharp. Unfortunately we could not find yet an explicit counterexample in the $W^{m,1}$ class. 
Our estimates are still valid for complex-valued potentials, if $E$ is sufficiently large with respect to $\|v\|_{C(\bar D)}$: in this case we can't use the formulas at the beginning of Section \ref{sechil} for the solution of the Riemann-Hilbert problem and thus it is necessary to follow a more general approach, like in \cite[\S 6]{N3}.

Estimate \eqref{est1} also extends the result obtained in \cite{S1} for the same problem at zero energy and in \cite{S2} at negative energy. In dimension $d \geq 3$ a global stability estimate similar to \eqref{est1} was proved in \cite{N2}, at zero energy.\smallskip

As regards \eqref{est2}, its main feature is the explicit dependence on the energy $E$. This estimate consist of two parts, the first logarithmic and the second H\"older; when $E$ increases, the logarithmic part decreases and the H\"older part becomes dominant. This estimate is sharp not only with respect to the dependence on the smoothness of the potentials, but also with respect to the energy, as shown in \cite{I2}. It extends the result of \cite{S2}, where a similar energy-dependent stability estimate was obtained at negative energy. Yet in that case the H\"older part grows exponentially with the energy, while in the present work it grows linearly. For this reason estimate \eqref{est2}, namely when $\tau = 1$, is totally coherent with the Lipschitz stable approximate reconstruction algorithms developed in \cite{N4} and \cite{NS2}.

Estimate \eqref{est2} is the first stability result in two dimensions for the Gel'fand-Calder\'on problem at positive energy with an explicit dependence on the smoothness of potential and on the energy. In dimension $d \geq 3$, global energy-dependent stability estimates changing from logarithmic type to Lipschitz type at high energies were given in \cite{Isa} and greatly improved in \cite{IN2}. In turn, the paper \cite{Isa} was preceeded by \cite{N7}. See also \cite{I3} for similar estimates for another inverse boundary value problem.\smallskip

The proof of Theorem \ref{maintheo} follows the scheme of \cite{S2} and it is based on $\bar \partial$ techniques.
The map $\Phi(E) \to v(x)$ is considered as the composition of $\Phi(E) \to (r(\lambda) , \rho(\lambda,\lambda'))$ and $(r(\lambda), \rho(\lambda,\lambda')) \to v(x)$, where $r(\lambda)$ and $\rho(\lambda,\lambda')$ are complex valued functions, closely related to the so-called generalised scattering amplitude (see Section \ref{sec2} for details).

The stability of $\Phi(E) \to (r(\lambda),\rho(\lambda,\lambda'))$ -- previously known only for $E \leq 0$ -- relies on some identities of \cite{N5} (based in particular on \cite{A}), and estimates on $r(\lambda)$ for $\lambda$ near $0$ and $\infty$. The reconstrution of $r(\lambda)$ from $\Phi(E)$ is logarithmic stable with respect to $\Phi$ (at fixed $E$), while the reconstruction of $\rho(\lambda,\lambda')$ is Lipschitz stable. These results are proved in section \ref{secphir}.

The stability of $(r(\lambda),\rho(\lambda,\lambda')) \to v(x)$ is of H\"older type and it is proved in section \ref{sechil}. This is the most challenging part of the paper because we need to establish several new estimates for the non-local Riemann-Hilbert problem solved by $r(\lambda)$ and $\rho(\lambda,\lambda')$ (see Section \ref{sec2} for details). We make great use of the theory of generalised analytic functions of Ahlfors-Vekua and the main reference is \cite{V}. In particular, we establish pointwise and $L^p$ estimates for solutions of non-homogeneous $\bar \partial$-equations with pole singularities.

In Section \ref{secpf} we show how the composition of the two above-mentioned maps gives the result of Theorem \ref{maintheo}.

\section{Preliminaries} \label{sec2}
We recall the definition of the Faddeev eigenfunctions $\psi(x,k)$ of equation \eqref{equa}, for $x=(x_1,x_2) \in \R^2$, $k=(k_1,k_2) \in \SE \subset \C^2$, $\SE = \{ k \in \C^2 : k^2 = k_1^2 + k_2^2 = E\}$ for $E \neq 0$ (see \cite{F}, \cite{N3}). We first extend $v \equiv 0$ on $\R^2 \setminus D$ and define $\psi(x,k)$ as the solution of the following integral equation:
\begin{align} \label{inteq}
\psi(x,k) &= e^{ikx}+\int_{y \in \R^2} G(x-y, k)v(y)\psi(y,k) dy,\\ \label{greenfad}
G(x,k)&=g(x,k)e^{ikx}, \\ \label{defgg}
g(x,k)&=-\left(\frac{1}{2 \pi}\right)^2 \int_{\xi \in \R^2} \frac{e^{i\xi x}}{\xi^2+2k\xi}d\xi,
\end{align}
where $x \in \R^2$, $k\in \SE \setminus \R^2$. It is convenient to write \eqref{inteq} in the following form
\begin{equation} \label{inteq2}
\mu(x,k) = 1 + \int_{y \in \R^2} g(x-y,k)v(y)\mu(y,k)dy,
\end{equation}
where $\mu(x,k)e^{ikx} = \psi(x,k)$.

For $\Imm k =0$ formulas \eqref{inteq}-\eqref{inteq2} make no sense; however, the following limits make sense
\begin{gather} \label{defpsig}
\psi_{\gamma}(x,k)=\psi(x,k+i0\gamma), \qquad G_{\gamma}(x,k)= G(x,k+i0\gamma),\\
\mu_{\gamma}(x,k)=\mu(x,k+i0\gamma),
\end{gather}

We define $\E_E \subset \SE \setminus \R^2$ the set of exceptional points of integral equation \eqref{inteq2}: $k \in \SE \setminus (\E_E \cup \R^2)$ if and only if equation \eqref{inteq2} is uniquely solvable in $L^{\infty}(\R^2)$.

\begin{rem} \label{reme}
From \cite[Proposition 1.1]{N4} we have that there exists $E_0 = E_0(\|v\|_{m,1},D)$ such that for $|E| \geq E_0(\|v\|_{m,1},D)$ there are no exceptional points for equation \eqref{inteq2}, i.e. $\E_E = \emptyset$: thus the Faddeev eigenfunctions exist (unique) for all $k \in \SE \setminus \R^2$.
\end{rem}

Following \cite{GM}, \cite{N3}, we make the change of variables
\begin{gather} \nonumber
z = x_1 + i x_2, \qquad \lambda = \frac{k_1+ik_2}{\sqrt{E}}, \\ \nonumber
k_1 = \left(\lambda + \frac{1}{\lambda}\right)\frac{\sqrt{E}}{2}, \qquad k_2 = \left(\frac{1}{\lambda}- \lambda\right)\frac{i\sqrt{E}}{2},
\end{gather}
and write $\psi, \mu$ as functions of these new variables. For $|\lambda|=1$ and $E>0$ formulas \eqref{defgg} and \eqref{inteq2} make no sense but the following limits do:
\begin{gather}
 \psi_{\pm}(z,\lambda) =\psi(z,\lambda(1\mp 0)), \qquad \mu_{\pm}(z,\lambda) = \psi(z,\lambda(1\mp 0)), \\ \label{defgpm}
g_{\pm}(z,\lambda)=g(z,\lambda(1\mp 0)).
\end{gather}

For $k \in \SE \setminus (\E_E \cup \R^2)$ we can define, for the corresponding $\lambda$, the following generalised scattering amplitude,
\begin{align} \label{defb}
b(\lambda,E) &= \frac{1}{(2 \pi)^2} \! \! \int_{\C} \exp \bigg[\frac i 2 \sqrt{E}\left(1+(\mathrm{sgn}\,E)\frac{1}{\lambda \bar \lambda} \right)\\ \nonumber
&\quad \times \left( (\mathrm{sgn}\,E)z\bar \lambda + \lambda \bar z\right)\bigg]v(z)\mu(z,\lambda)d\Ree z\,d\Imm z,
\end{align}
and the functions $h_{\pm}$,
\begin{align} \label{defhpm}
h_{\pm}(\lambda,\lambda',E)&=\left(\frac{1}{2 \pi}\right)^2\int_{\C}\exp\left[-\frac{i}{2}\sqrt{E}(\lambda'\bar z+z/\lambda')\right]\\ \nonumber
&\quad \times v(z)\psi_{\pm}(z,\lambda)d\Ree z \, d\Imm z,
\end{align}
for $|\lambda|=|\lambda'|=1$. It is useful to introduce the following auxiliary functions $h_1, h_2$,
\begin{align} \label{defh1}
h_1(\lambda,\lambda')&=\theta\left[-\frac{1}{i}\left(\frac{\lambda'}{\lambda}-\frac{\lambda}{\lambda'}\right) \right]h_+(\lambda,\lambda')\\ \nonumber
&\quad -\theta\left[\frac{1}{i}\left(\frac{\lambda'}{\lambda}-\frac{\lambda}{\lambda'}\right) \right]h_-(\lambda,\lambda'),\\ \label{defh2}
h_2(\lambda,\lambda')&=\theta\left[-\frac{1}{i}\left(\frac{\lambda'}{\lambda}-\frac{\lambda}{\lambda'}\right) \right]h_-(\lambda,\lambda')\\ \nonumber
&\quad -\theta\left[\frac{1}{i}\left(\frac{\lambda'}{\lambda}-\frac{\lambda}{\lambda'}\right) \right]h_+(\lambda,\lambda'),
\end{align}
and $\rho$, solution of the following integral equations,
\begin{subequations}\label{defrho}
\begin{align} \label{defrho1}
&\rho(\lambda,\lambda')+\pi i\int_{|\lambda''|=1}\rho(\lambda,\lambda'')\theta\left[\frac{1}{i}\left(\frac{\lambda'}{\lambda''}-\frac{\lambda''}{\lambda'}\right) \right]\\ \nonumber
&\qquad \times h_1(\lambda'',\lambda')|d\lambda''|=-\pi ih_1(\lambda,\lambda'),\\
&\rho(\lambda,\lambda')+\pi i\int_{|\lambda''|=1}\rho(\lambda,\lambda'')\theta\left[-\frac{1}{i}\left(\frac{\lambda'}{\lambda''}-\frac{\lambda''}{\lambda'}\right) \right]\\ \nonumber
&\qquad \times h_2(\lambda'',\lambda')|d\lambda''|=-\pi ih_2(\lambda,\lambda'),
\end{align}
\end{subequations}
for $|\lambda|=|\lambda'|=1$. Here and in the following we drop the dependence of some functions on $E$ for simplicity's sake.

The functions just defined play an important role in the following Riemann-Hilbert problem solved by $\mu$. When $v$ is real-valued and $E>0$ we have (see \cite{N3} for more details):
\begin{equation}\label{dbar}
\frac{\partial}{\partial \bar \lambda}\mu(z,\lambda) = r(z,\lambda)\overline{\mu(z,\lambda)},
\end{equation}
for $\lambda$ not an exceptional point (i.e. $k(\lambda) \in \SE \setminus (\E_E \cup \R^2)$) and $|\lambda| \neq 1$, where
\begin{align} \label{defr2}
r(z,\lambda) &= r(\lambda)\exp \bigg[-\frac i 2 \sqrt{E}\left(1+(\mathrm{sgn}\,E)\frac{1}{\lambda \bar \lambda} \right) \left( z\bar \lambda + \lambda \bar z\right)\bigg],\\  \label{defr}
r(\lambda) &= \frac{\pi}{\bar \lambda}\mathrm{sgn}(\lambda \bar \lambda -1) b(\lambda,E),
\end{align}
where $b$ is defined in \eqref{defb};
\begin{align}\label{corrt}
\mu_+(z,\lambda) = \mu_-(z,\lambda)+\int_{|\lambda'|=1}\rho(\lambda,\lambda',z)\mu_-(z,\lambda')|d\lambda'|,
\end{align}
for $|\lambda|=1$, where
\begin{equation}\label{defrhoz}
\rho(\lambda,\lambda',z)= \rho(\lambda,\lambda')\exp\left[\frac{i\sqrt{E}}{2}\left((\lambda'-\lambda)\bar z+\left(\frac{1}{\lambda'}-\frac{1}{\lambda}\right)z\right)\right],
\end{equation}
where $\rho(\lambda,\lambda')$ is defined in \eqref{defrho}. In addition we have
\begin{gather}
\lim_{|\lambda|\to \infty} \mu(z,\lambda)=1,\\
\mu(z,\lambda)=1+\mu_{-1}(z)\lambda^{-1}+o(|\lambda|^{-1}),\quad \text{for } |\lambda| \to \infty,\\ \label{vmu1}
v(z)=2i\sqrt{E}\frac{\partial}{\partial z}\mu_{-1}(z).
\end{gather}

We recall that if $v \in W^{m,1}(\R^2)$ with $\supp v \subset D$, then $\|\hat{v} \|_{\alpha,m} < +\infty$ for some $0 < \alpha<1$, where
\begin{gather}
\hat v(p) = (2 \pi)^{-2} \int_{\R^2} e^{ipx} v(x) dx, \qquad p \in \C^2, \\
\|u\|_{\alpha, m} = \| (1+|p|^2)^{m/2} u(p)\|_{\alpha}, \\
\|w\|_{\alpha} = \sup_{p,\xi \in \R^2, |\xi|\leq 1} \left(|w(p)|+|\xi|^{-\alpha}|w(p+\xi)-w(p)|\right),
\end{gather}
for test functions $u,w$.

We restate in an adapted form a lemma from \cite[Lemma 2.1]{N4}.
\begin{lem} \label{lem21}
Let the conditions \eqref{cv1}, \eqref{cv2} hold for a potentials $v$. Let $\mu(x,k)$ be the associated Faddeev functions. Then, for any $0<\sigma<1$, we have
\begin{align}
|\mu(x,k)-1|+\left|\frac{\partial \mu(x,k)}{\partial x_1}\right|+\left|\frac{\partial \mu(x,k)}{\partial x_2}\right| \leq |\Ree\, k|^{-\sigma}c(m,\sigma) \|\hat{v}\|_{\alpha,m}, 
\end{align} 
for $k \in \C^2\setminus \R^2$,
\begin{align}
|\mu_{\gamma}(x,k)-1|+\left|\frac{\partial \mu_{\gamma}(x,k)}{\partial x_1}\right|+\left|\frac{\partial \mu_{\gamma}(x,k)}{\partial x_2}\right| \leq |k|^{-\sigma}c(m,\sigma) \|\hat{v}\|_{\alpha,m}, 
\end{align} 
for $k \in \R^2$, $\gamma \in S^1$. In both cases we suppose also that $k^2 \geq R$, where $R$ is defined in Lemma \ref{lemb}.
\end{lem}

The following lemma is a variation of a result in \cite{N4} and it is proved in \cite{S2}.
\begin{lem} \label{lemb}
Let the conditions \eqref{cv1}, \eqref{cv2} hold for a potentials $v$ and let $E \in \R \setminus \{0\}$. Then there exists an $R = R(m,\|\hat{v} \|_{\alpha,m}) > 1$, such  that
\begin{equation}
|b(\lambda,E)| \leq 2 \|\hat{v}\|_{\alpha,m} \left( 1+|E|\left(|\lambda|+\mathrm{sgn}(E)/|\lambda| \right)^2 \right)^{-m/2},
\end{equation}
for $|\lambda| > \frac{2R}{|E|^{1/2}}$ and $|\lambda|<\frac{|E|^{1/2}}{2R}$
\end{lem}

Let us mention that Lemma 2.2 of \cite{S1} and Lemma 2.1 of \cite{S2} should be corrected using the norm $\| \cdot \|_{\alpha,m}$ instead of $\|\cdot \|_{m}$.

We also restate \cite[Lemma 2.6]{BBR}. 
\begin{lem}[\cite{BBR}] \label{lemtech}
Let $q_1\in L^{s_1}(\C) \cap L^{s_2}(\C)$, $1 < s_1 <2 < s_2 < \infty$ and $q_2 \in L^s( \C)$, $1 < s <2$. Assume $u$ is a function in $L^{\tilde s}(\C)$, with $1/\tilde s = 1/s - 1/2$, which satisfies
\begin{equation}
\frac{\partial u (\lambda)}{\partial \bar \lambda} = q_1(\lambda) \bar u(\lambda) + q_2(\lambda), \qquad \lambda \in \C.
\end{equation}
Then there exists $c=c(s,s_1,s_2) >0$ such that
\begin{equation}
\|u\|_{L^{\tilde s}} \leq c \|q_2\|_{L^s} \exp(c (\|q_1\|_{L^{s_1}}+\|q_1\|_{L^{s_2}})).
\end{equation}
\end{lem}

We will make also use of the well-known H\"older's inequality, which we recall in a special case: for $f \in L^p(\C)$, $g \in L^q(\C)$ such that $1 \leq p,q\leq \infty$, $1\leq r < \infty$, $1/p+1/q = 1/r$, we have
\begin{equation} \label{holder}
 \|fg\|_{L^r(\C)}\leq \|f\|_{L^p(\C)}\|g\|_{L^q(\C)}.
\end{equation}

Throughout all the paper $c(\alpha, \beta, \ldots)$ is a positive constant depending on parameters $\alpha, \beta, \ldots$

\section{From $\Phi(E)$ to $r(\lambda)$ and $\rho(\lambda,\lambda')$} \label{secphir}
We begin recalling a lemma from \cite{S2}, which we restate in the case $E >0$.
\begin{lem} \label{lemestrr}
Let the conditions \eqref{cv1}, \eqref{cv2} hold and take $0<a_1 \leq \min\left(1, \frac{|E|^{1/2}}{2R}\right)$, $a_2 \geq \max\left(1,\frac{2R}{|E|^{1/2}}\right)$,
for $E > 0$ and $R$ as defined in Lemma \ref{lemb}. Then for $p \geq 1$ we have
\begin{align} \label{estlem21}
\left\||\lambda|^j r(\lambda)\right\|_{L^p(|\lambda|<a_1)} &\leq c(p,m) \|\hat{v}\|_{\alpha,m} |E|^{-m/2} a_1^{m-1+j+2/p}, \\ \label{estlem22}
\left\||\lambda|^j r(\lambda)\right\|_{L^p(|\lambda|>a_2)} &\leq c(p,m) \|\hat{v}\|_{\alpha,m} |E|^{-m/2} a_2^{-m-1+j+2/p},
\end{align}
where $j=1,0,-1$ and $r$ was defined in \eqref{defr}.
\end{lem}
Note that, in contrast to the case $E <0$, this Lemma holds even when $a_1 = a_2 =1$, thanks to the sign in Lemma \ref{lemb}.

The following Lemma extends \cite[Lemma 3.2]{S2} to the positive energy case.
\begin{lem} \label{lemdifh}
Let $D \subset \{ x \in \R^2 \, : \, |x| \leq l\}$, $E > 0$, $v_1, v_2$ be two potentials satisfying \eqref{direig}, \eqref{cv1}, \eqref{cv2}, $\Phi_1(E), \Phi_2(E)$ the corresponding Dirichlet-to-Neumann operator and $b_1, b_2$ the corresponding generalised scattering amplitude. Let $\|v_j\|_{m,1} \leq N$, $j=1,2$. Then we have
\begin{align} \label{estdifh}
|b_2(\lambda) - b_1(\lambda)|\leq c(D,N,m)\exp\left[l\sqrt{|E|}\left||\lambda|-\frac{1}{|\lambda|}\right|\right]\|\Phi_2(E) - \Phi_1(E)\|_*,
\end{align}
for $\lambda \neq 0$, where $\| \cdot\|_* = \|\cdot \|_{L^{\infty}(\partial D)\to L^{\infty}(\partial D)}$.
\end{lem}
\begin{proof}
We have the following identity:
\begin{align} \label{aless}
b_2(\lambda) - b_1(\lambda) = \left(\frac{1}{2\pi}\right)^2\int_{\partial D}\psi_1(x,\overline{k(\lambda)})(\Phi_2(E) - \Phi_1(E))\psi_2(x,k(\lambda)) dx,
\end{align}
where $\psi_i(x,k)$ are the Faddeev functions associated to the potential $v_i$, $i=1,2$. This identity is a particular case of the one in \cite[Theorem 1]{N5}: we refer to that paper for a proof.

From this identity we obtain:
\begin{align} \label{estlem1}
|b_2(\lambda) - b_1(\lambda)| \leq \frac{1}{(2\pi)^2} \|\psi_1(\cdot,k)\|_{L^{\infty}(\partial D)}\|\Phi_2(E) - \Phi_1(E)\|_* \|\psi_2(\cdot,k)\|_{L^{\infty}(\partial D)}.
\end{align}
Now using Lemma \ref{lem21} and the change of variables in Section \ref{sec2}, we get
\begin{align*}
&\|\psi_j(\cdot,k(\lambda))\|_{L^{\infty}(\partial D)} \leq \|e^{\frac{i\sqrt{E}}{2} (\bar z\lambda + z / \lambda)} \mu_j(\cdot,k(\lambda))\|_{L^{\infty}(\partial D)}\\
&\quad \leq e^{\frac{\sqrt{E}}{2}l |\lambda -1/ \bar \lambda|} \|\mu_j(\cdot,k(\lambda))\|_{L^{\infty}(\partial D)}\\
&\quad \leq e^{\frac{\sqrt{E}}{2}l \left||\lambda| -|1/ \lambda|\right|} \left(\|\mu_j(\cdot,k(\lambda)) -1\|_{L^{\infty}(\partial D)} +\|1\|_{L^{\infty}(\partial D)}\right)\\
&\quad \leq c(D,N,m) e^{\frac{\sqrt{E}}{2}l \left||\lambda| -|1/ \lambda|\right|},
\end{align*}
for $j=1,2$. This, combined with \eqref{estlem1}, gives \eqref{estdifh}.
\end{proof}

The following proposition shows that the map $\Phi(E) \to \rho(\lambda,\lambda')$ is Lipschitz stable.
\begin{prop} \label{proprho}
Let $D \subset \{ x \in \R^2 \, : \, |x| \leq l\}$, $E > 0$, $v_1, v_2$ be two potentials satisfying \eqref{direig}, \eqref{cv1}, \eqref{cv2}, $\Phi_1(E), \Phi_2(E)$ the corresponding Dirichlet-to-Neumann operator and $\rho_1, \rho_2$ the corresponding functions as defined in \eqref{defrho}. Let $\|v_j\|_{m,1} \leq N$, $j=1,2$. Then we have
\begin{align} \label{estdifrho}
\|\rho_2-\rho_1\|_{L^2(T \times T)} \leq c(D,N,m) \|\Phi_2(E) - \Phi_1(E)\|_*,
\end{align}
for $E \geq E_2 = E_2(N,m)$, where $T = \{ \lambda \in \C \, : \, |\lambda|=1 \}$ and $\| \cdot\|_* = \|\cdot \|_{L^{\infty}(\partial D)\to L^{\infty}(\partial D)}$.
\end{prop}
\begin{proof}
We begin proving 
\begin{align}  \label{estdiff}
 \|f_2 - f_1\|_{L^2(T \times T)} \leq c(D,N,m) \|\Phi_2(E) - \Phi_1(E)\|_*,
\end{align}
where $f_j$ is the scattering amplitude related to potential $v_j$, $j=1,2$, defined as
\begin{align} \nonumber
f_j(\lambda,\lambda')=\left(\frac{1}{2\pi}\right)^2\int_{\C}\exp\left[-\frac{i\sqrt{E}}{2}\left(\lambda'\bar z+ \frac{z}{\lambda'}\right)\right] v_j(z) \varphi_j^+(z,\lambda)d\Ree z\, d\Imm z.
\end{align}
Here we used the change of variables in Section 2, and $\varphi_j^+(x,k) = \psi_{k/|k|}(x,k)$, where $\psi_{\gamma}(x,k)$ was defined in \eqref{defpsig}.
The following identity holds:
\begin{align} \nonumber
&f_2(\lambda,\lambda')-f_1(\lambda,\lambda')\\ \nonumber
&\quad =\left(\frac{1}{2\pi}\right)^2\int_{\partial D}\varphi^+_1(x,-k(\lambda'))(\Phi_2(E) - \Phi_1(E))\varphi^+_2(x,k(\lambda)) dx,
\end{align}
for $|\lambda|=|\lambda'|=1$. This is proved in \cite[Theorem 1]{N5}. We then obtain
\begin{align} \label{estlem2}
&|f_2(\lambda,\lambda')-f_1(\lambda,\lambda')|\\ \nonumber
&\quad \leq \frac{1}{(2\pi)^2} \|\varphi^+_1(\cdot,k)\|_{L^{\infty}(\partial D)}\|\Phi_2(E) - \Phi_1(E)\|_* \|\varphi^+_2(\cdot,k)\|_{L^{\infty}(\partial D)},
\end{align}
where $|\lambda|=|\lambda'|=1$ and so $k \in \R^2$, $k^2=E$. From Lemma \ref{lem21} we get
\begin{align*}
&\|\varphi^+_j(\cdot,k)\|_{L^{\infty}(\partial D)} = \|\mu_{k/|k|}(\cdot,k)\|_{L^{\infty}(\partial D)} \leq c(D,N,m),
\end{align*}
for $j=1,2$, since $k\in \R^2$, $k^2=E$. This, combined with \eqref{estlem2}, gives \eqref{estdiff}.\smallskip

It is now useful to recall the following integral equations which relate $f_j$ with $h^j_{\pm}$ (see \cite{F2, N3}):
\begin{align} \label{relhrho}
&h^{j}_{\pm}(\lambda, \lambda') - \pi i \int_{|\lambda''|=1}h^{j}_{\pm}(\lambda,\lambda'')\theta\left[ \pm \frac{1}{i}\left(\frac{\lambda''}{\lambda}-\frac{\lambda}{\lambda''}\right)\right] \\ \nonumber
&\qquad \times f_{j}(\lambda'',\lambda') |d\lambda''| = f_{j}(\lambda,\lambda') \qquad j=1,2.
\end{align}
Subtracting this equation for $j=2$ and $j=1$ we obtain
\begin{align} \label{eqhf}
\left(I+P^2_{\pm}\right)\left(h^2_{\pm}-h^1_{\pm}\right)  = \left(I+Q^1_{\pm}\right)\left(f_2-f_1\right),
\end{align}
where
\begin{align}
(P^j_{\pm}u)(\lambda,\lambda') &= -\pi i \int_{\lambda'' \in T}\! \! \! \! \! \! \! \! \! \! u(\lambda, \lambda'')\theta\left[ \pm \frac{1}{i}\left(\frac{\lambda''}{\lambda}-\frac{\lambda}{\lambda''}\right)\right] f_{j}(\lambda'',\lambda') |d\lambda''|,\\
(Q^j_{\pm}u)(\lambda,\lambda')&=\pi i \int_{\lambda'' \in T} \! \! \! \! \! \! \! \! \! \! h^{j}_{\pm}(\lambda,\lambda'')\theta\left[ \pm \frac{1}{i}\left(\frac{\lambda''}{\lambda}-\frac{\lambda}{\lambda''}\right)\right] u(\lambda,\lambda')|d\lambda''|,
\end{align}
for $u \in L^p(T^2)$, $p > 1$. In \cite[\S 2]{N4} it is proved that 
\begin{align}
|f_j(\lambda,\lambda')| \leq 2 \|\hat v_j\|_{\alpha,m} (1+E|\lambda-\lambda'|^2)^{-m/2},\\ \label{decrh}
|h^j_{\pm}(\lambda,\lambda')| \leq 2 \|\hat{v_j}\|_{\alpha,m} (1+E|\lambda-\lambda'|^2)^{-m/2},
\end{align}
for $\lambda, \lambda' \in T$ and $E \geq E_1 = E_1(N,D,m)$. From these inequalities (and also inequalities (2.45) of \cite{N4}) we find that
\begin{align}
 \|P^j_{\pm}u\|_{L^2(T\times T)} \leq \frac{c_1(N,m)}{E^{1/4}} \|u\|_{L^2(T\times T)},\\
 \|Q^j_{\pm}u\|_{L^2(T\times T)} \leq \frac{c_2(N,m)}{E^{1/4}} \|u\|_{L^2(T\times T)}.
\end{align}
Choose $E'_1 \geq E_1$ such that $\max\left(\frac{c_1(N,m)}{{E'}_1^{1/4}},\frac{c_2(N,m)}{{E'}_1^{1/4}}\right) \leq \frac 1 2$. Then $P^j_{\pm}$ is invertible on $L^2(T\times T)$ and from \eqref{eqhf} we obtain
\begin{equation}\label{hpmf}
 \|h^2_{\pm}-h^1_{\pm}\|_{L^2(T \times T)} \leq c_3(N,D,m) \|f_2-f_1 \|_{L^2(T \times T)},
\end{equation}
for $E \geq E'_1$.
It is straightforward to see that
\begin{equation}\label{hgammahpm}
 \|h^2_{\beta}-h^1_{\beta}\|_{L^2(T \times T)} \leq \|h^2_{+}-h^1_{+}\|_{L^2(T \times T)} +\|h^2_{-}-h^1_{-}\|_{L^2(T \times T)},\; \beta =1,2,
\end{equation}
where $h^j_{\beta}$ are the auxiliary functions defined in \eqref{defh1} and \eqref{defh2} related to the potential $v_j$, $j=1,2$.

In order to finish the proof, it is sufficient to remark that the functions $h^j_{\beta}$ satisfy inequality \eqref{decrh}, as well as $\rho_j$. For this, we will need the following lemma, which will be proved later.
\begin{lem} \label{lemrho}
The function $\rho$, defined in \eqref{defrho} for a potential $v$ such that $\|v\|_{m,1} \leq N$, satisfies the inequality
 \begin{equation} \label{decrrho}
 |\rho(\lambda,\lambda')| \leq c(N) (1+E|\lambda-\lambda'|^2)^{-m/2},
\end{equation}
for $\lambda, \lambda' \in T$ and $E \geq \tilde E_1 = \tilde E_1(N,D,m)$.
\end{lem}

Now we see that thanks to Lemma \ref{lemrho}, equations \eqref{defrho} have the same structure of equations \eqref{relhrho}, and the kernels satisfy the same inequalities for $E \geq E_2 (\tilde E_1,E'_1)$. Thus we obtain directly
\begin{equation} \label{rhohgamma}
 \|\rho_2-\rho_1\|_{L^2(T \times T)} \leq c_4(N,D,m) \|h^2_{\beta}-h^1_{\beta}\|_{L^2(T \times T)}, \beta =1,2,
\end{equation}
for $E \geq E_2$.

Now inequalities \eqref{rhohgamma}, \eqref{hgammahpm}, \eqref{hpmf} together with \eqref{estdiff} give \eqref{estdifrho}, which finishes the proof.
\end{proof}

\begin{proof}[Proof of Lemma \ref{lemrho}]
 We write the integral equation defining $\rho$, \eqref{defrho1}, as follows:
\begin{equation} \label{eqqrho}
(I+H_1)\rho(\lambda, \lambda') = -\pi i h_1(\lambda,\lambda'),
\end{equation}
where
\begin{equation}
H_1 \rho(\lambda,\lambda')= \pi i\int_{|\lambda''|=1}\rho(\lambda,\lambda'')\theta\left[\frac{1}{i}\left(\frac{\lambda'}{\lambda''}-\frac{\lambda''}{\lambda'}\right) \right]h_1(\lambda'',\lambda')|d\lambda''|.
\end{equation}
We want to prove that this equation has a unique solution in the space of complex-valued functions $g(\lambda, \lambda')$ defined on $T^2$, such that 
\begin{equation} \label{ineg}
|g(\lambda, \lambda')| \leq c (1+E|\lambda-\lambda'|^2)^{-m/2},
\end{equation}
for some constant $c$. Let us call this function space $S$ and define $\|g\|_S = \inf c$ such that \eqref{ineg} is verified.

We have that $\|H_1 g\|_S \leq c E^{-1/2} \|g\|_S $. Indeed, since $h_1$ satisfies inequality \eqref{decrh}, the following estimate holds
\begin{align}
|H_1 g (\lambda,\lambda')| \leq c \|g\|_S\int_{T}\frac{|d\lambda''|}{\left((1+E|\lambda-\lambda''|^2)(1+E|\lambda''-\lambda'|^2)\right)^{m/2}}.
\end{align}
We split the circle $T$, at fixed $\lambda, \lambda'$ in two sets: the first contains the points $\lambda''$ that are closer to $\lambda$ than to $\lambda'$, i.e. $|\lambda'' - \lambda|\leq |\lambda'' -\lambda'|$ and the second is the complement. For $\lambda''$ in the first set we have that $|\lambda'' - \lambda'| \geq \frac 1 2 |\lambda' - \lambda|$ while for $\lambda''$ in the second $|\lambda'' - \lambda| \geq \frac 1 2 |\lambda' - \lambda|$. Thus we obtain
\begin{align}
|H_1 g (\lambda,\lambda')| \leq c \frac {\|g\|_S}{(2+E|\lambda-\lambda'|^2)^{m/2}}\int_{T}\frac{|d\lambda''|}{(1+E|\lambda''- \tilde\lambda|^2)^{m/2}},
\end{align}
where $\tilde \lambda$ is some point in $T$. Using inequality \eqref{bella} we obtain the estimate for $H_1$.
Then for sufficiently large $E$, equation \eqref{eqqrho} has a unique solution in $S$ by iteration. This finishes the proof of Lemma \ref{lemrho}.
\end{proof}

In the following proposition we prove that the map $\Phi(E) \to r(\lambda)$ is logarithmic stable.
\begin{prop}\label{propestr}
Let $E$ be such that $E\geq  E_3 = \max ((2R)^2, E_0)$, where $R$ is defined in Lemma \ref{lemb} and $E_0$ in Remark \ref{reme}, let $v_1, v_2$ be two potentials satisfying \eqref{direig}, \eqref{cv1}, \eqref{cv2}, $\Phi_1(E), \Phi_2(E)$ the corresponding Dirichlet-to-Neumann operator and $r_1, r_2$ as defined in \eqref{defr}. Let $\|v_k\|_{m,1} \leq N$, $k=1,2$. Then for every $p \geq 1 $ there exists a constant $\theta = \theta(D,N,m,p)$ such that for any $0 \leq \kappa < \frac{1}{4(l+1)}$, where $l = \mathrm{diam}(D)$, and for $E \geq E_3$ we have
\begin{align} \label{mainesth}
\left\|\left(|\lambda|+\frac{1}{|\lambda|}\right) |r_2 - r_1|\right\|_{L^p(\C)} \! \! \! \! \! \! \! \! &\leq \theta \bigg[ E^{-1}\left(E^{1/2}+\kappa \log(3+\delta^{-1})\right)^{-(m-2)} \\ \nonumber
&\quad + \frac{\delta(3+\delta^{-1})^{4\kappa(l+1)}}{E^{1/2p}}\bigg],
\end{align}
where $\delta = \|\Phi_2(E) - \Phi_1(E)\|_{L^{\infty}(\partial D)\to L^{\infty}(\partial D)}$.
\end{prop}

\begin{proof}
We choose $0<a_1 \leq 1 \leq a_2$ to be determined and split down the left hand side of \eqref{mainesth} as follows:
\begin{align*}
&\left\|\left(|\lambda|+\frac{1}{|\lambda|}\right) |r_2 - r_1|\right\|_{L^p(\C)} \leq I_1 + I_2 + I_3, \\
&\qquad I_1 =\left\|\left(|\lambda|+\frac{1}{|\lambda|}\right) |r_2 - r_1|\right\|_{L^p(|\lambda| <a_1)}, \\
&\qquad I_2 = \left\|\left(|\lambda|+\frac{1}{|\lambda|}\right) |r_2 - r_1|\right\|_{L^p(a_1<|\lambda| <a_2)},\\
&\qquad I_3 = \left\|\left(|\lambda|+\frac{1}{|\lambda|}\right) |r_2 - r_1|\right\|_{L^p(|\lambda| >a_2)}.
\end{align*}

From \eqref{estlem21} and \eqref{estlem22} we have
\begin{align}\label{esta}
I_1 &\leq c(N,p,m)  E^{-m/2} a_1^{m-2+2/p}, \\ \label{estb}
I_3 &\leq c(N,p,m)  E^{-m/2} a_2^{-m+2/p}
\end{align}

Lemma \ref{lemdifh} yields that $I_2$ can be estimated from above by
\begin{align} \label{estc}
c(D,N,p)\frac{\delta}{E^{1/2p}} \left(e^{4(l+1)\sqrt{E}\left(\frac{1}{a_1}-1\right)}+ e^{4(l+1)\sqrt{E}(a_2-1)}\right).
\end{align}
Here we used the fact that
\begin{align} \label{estexp}
 \left(\frac{1}{|\lambda|}+|\lambda|\right)\frac{e^{l\sqrt{E}\left||\lambda|-1/|\lambda|\right|}}{|\lambda|}&\leq e^{2l\sqrt{E}\left||\lambda|-1/|\lambda|\right|}\\ \nonumber
&\leq e^{4l\sqrt{E}(1/|\lambda|-1)}\chi_{|\lambda|<1}+e^{4l\sqrt{E}(|\lambda|-1)}\chi_{|\lambda|>1}
\end{align}
where $\chi_A$ is the characteristic function of a set $A \subset \C$.

Now we define, in \eqref{esta}-\eqref{estc},
\begin{equation}
a_2 = \frac{1}{a_1}=1 + \frac{\kappa \log (3+\delta^{-1})}{\sqrt{E}},
\end{equation}
for $0\leq \kappa<\frac{1}{4(l+1)}$. Note that $a_2 \geq 1 $ and $a_1 \leq 1$. Then we obtain, for every $p \geq 1$,
\begin{align} \label{estaa}
&I_j \leq c(N,p,m)E^{-1}(\sqrt{E}+\kappa \log (3+\delta^{-1}))^{-(m-2)}, \quad j=1,3,
\end{align}
To estimate $I_2$ we remark that
\begin{align}\label{estcc}
e^{4(l+1)\sqrt{E}\left(\frac{1}{a_1}-1\right)}+ e^{4(l+1)\sqrt{E}(a_2-1)}&=2e^{2(l+1)\kappa \log(3+\delta^{-1})} \\ \nonumber
&= 2(3+\delta^{-1})^{4(l+1)\kappa}.
\end{align}
Putting \eqref{estaa}-\eqref{estcc} together we find
\begin{align*}
\left\|\left(|\lambda|+\frac{1}{|\lambda|}\right) |r_2 - r_1|\right\|_{L^p(\C)} \! \! \! \! \! \! \! \! &\leq \theta_2 \bigg[ E^{-1}\left(E^{1/2}+\kappa \log(3+\delta^{-1})\right)^{-(m-2)} \\ \nonumber
&\quad + \frac{\delta(3+\delta^{-1})^{4(l+1)\kappa}}{E^{1/2p}}\bigg],
\end{align*}
which is estimate \eqref{mainesth}.
\end{proof}
\begin{rem}
In the following sections we will often implicitly use the basic fact that
\begin{equation} \nonumber
\| r_2 -r_1\|_{L^p(\C)} \leq \left\|\left(|\lambda|+\frac{1}{|\lambda|}\right) |r_2 - r_1|\right\|_{L^p(\C)}.
\end{equation}
\end{rem}

\section{Estimates for the non-local Riemann-Hilbert problem} \label{sechil}
We begin with an explicit formula relating a potential $v$, satisfying the assumption of Theorem \ref{maintheo}, with its associated functions $r(\lambda)$ and $\rho(\lambda,\lambda')$. This procedure allows us to explicitly solve the non-local Riemann-Hilbert problem presented in Section \ref{sec2} (see \eqref{dbar}--\eqref{defrhoz}).

The starting point is formula \eqref{vmu1}:
\begin{align}\label{defvu}
v(z)=2i\sqrt{E}\frac{\partial}{\partial z}\mu_{-1}(z),\\ \label{defmumeno}
\mu_{-1}(z)=\lim_{\lambda \to \infty}\lambda \left(\mu(z,\lambda)-1\right).
\end{align}
We follow the scheme of \cite[Theorem 6.1]{N3} in order to make $\mu$ explicit. In the following equations we omit the variable $z$ in the functions $\mu, e, K, \Omega, X$ for simplicity's sake. We have
\begin{align} \label{muex}
\mu(\lambda) = e(\lambda) +\frac{1}{2 \pi i}&\int_{|\zeta|=1}\Omega_1(\lambda,\zeta)K(\zeta)d\zeta-\Omega_2(\lambda,\zeta)\overline{K(\zeta)}d \bar \zeta,\\ \label{defe}
e(\lambda) &= 1-\frac{1}{\pi}\int_{\C}\frac{r(\zeta,z)\overline{e(\zeta)}}{\zeta-\lambda}d\Ree \zeta\, d\Imm \zeta,\\
\Omega_1(\lambda,\zeta)&=X_1(\lambda,\zeta)+iX_2(\lambda,\zeta),\\
\Omega_2(\lambda,\zeta)&=X_1(\lambda,\zeta)-iX_2(\lambda,\zeta),\\ \label{defX1}
X_1(\lambda,\zeta)+ \frac{1}{\pi}&\int_{\C}\frac{r(\eta,z)\overline{X_1(\eta,\zeta)}}{\eta-\lambda}d\Ree \eta\, d\Imm \eta = \frac{1}{2(\zeta-\lambda)},\\ \label{defX2}
X_2(\lambda,\zeta)+ \frac{1}{\pi}&\int_{\C}\frac{r(\eta,z)\overline{X_2(\eta,\zeta)}}{\eta-\lambda}d\Ree \eta\, d\Imm \eta = \frac{1}{2i(\zeta-\lambda)},\\ \label{defK}
K(\lambda)&=\mu_+(\lambda)-\mu_-(\lambda)\\ \nonumber
&= \int_{|\lambda'|=1}\rho(\lambda,\lambda',z)\Big[e(\lambda')\\ \nonumber
&\quad +\frac{1}{2\pi i}\int_{|\zeta|=1}\Omega_1(\lambda'(1+0),\zeta)K(\zeta)d\zeta\\ \nonumber
&\quad +\Omega_2(\lambda',\zeta)\overline{K(\zeta)}d\bar\zeta\Big]|d\lambda'|.
\end{align}
Let, for $p \geq 1, \nu \geq 0$,  $L^p_{\nu}(\C)$ be the function space
\begin{equation} \label{deflpnu}
\{f : \C\to\C \,| f(z) \in L^p(|z|\leq 1), |z|^{-\nu}f\left(\frac{1}{|z|}\right)\in L^p(|z|\leq 1)\},
\end{equation}
with the corresponding norm $\|f\|_{L^p_{\nu}(\C)} = \|f\|_{L^p(|z|\leq 1)}+\| |z|^{-\nu} f(1/|z|)\|_{L^p(|z|\leq 1)}$. From Lemma \ref{lemb} we have that $r_j \in L^p_{\nu}(\C)$ for all $\nu < m$. Then, from results of \cite{V}, equations \eqref{defe}, \eqref{defX1} and \eqref{defX2} are uniquely solved in $L^q_0(\C)$, $p/(p-1) \leq q <2$, and $e(\lambda)$ is continuous on $\C$. 

Then we can write
\begin{align} \label{mu-1}
 &\mu_{-1}(z) = \frac{1}{\pi}\int_{\C}r(\zeta,z)\overline{e(\zeta)}d\Ree \zeta\, d\Imm \zeta\\ \nonumber
&\quad +\frac{1}{2\pi i}\int_{|\zeta|=1}K(\zeta)\left[-1+\frac{1}{\pi}\int_{\C}r(\lambda,z)\overline{\Omega_2(\lambda,\zeta)}d\Ree \lambda \, d \Imm \lambda\right]d\zeta\\ \nonumber
&\quad -\frac{1}{2\pi i}\int_{|\zeta|=1}\overline{K(\zeta)}\left[\frac{1}{\pi}\int_{\C}r(\lambda,z)\overline{\Omega_1(\lambda,\zeta)}d\Ree \lambda \, d \Imm \lambda\right]d \bar \zeta,
\end{align}
where $\mu_{-1}$ was defined in \eqref{defmumeno}. 
Indeed, by Lebesgue's dominated convergence (using Lemma \ref{lemb}), we can calculate the following limits:
\begin{align*}
&\lim_{\lambda \to \infty}\lambda(e(\lambda)-1) =  \frac{1}{\pi}\int_{\C}r(\zeta,z)\overline{e(\zeta)}d\Ree \zeta\, d\Imm \zeta,\\
&\lim_{\lambda \to \infty}\lambda X_1(\lambda,\zeta) = -\frac{1}{2}+\frac{1}{\pi}\int_{\C}r(\eta,z)\overline{X_1(\eta,\zeta)}d\Ree\eta \, d \Imm \eta,\\
&\lim_{\lambda \to \infty}\lambda X_2(\lambda,\zeta) = -\frac{1}{2i}+\frac{1}{\pi}\int_{\C}r(\eta,z)\overline{X_2(\eta,\zeta)}d\Ree\eta \, d \Imm \eta.
\end{align*}
Then, in connection with \eqref{defvu}, we need to take the derivative of \eqref{mu-1} with respect to $\partial / \partial z = \partial_z$. We find:
\begin{gather} \nonumber
\partial_z \mu_{-1} = A + B+C,
\end{gather}
where
\begin{align}
A &= \frac{1}{\pi}\int_{\C}r(\zeta,z)\left[-\frac{i}{2}\sqrt{E}\left(\frac{1}{\zeta}+\bar \zeta\right)\overline{e(\zeta)}+\overline{\partial_{\bar z}e(\zeta)}\right]d\Ree \zeta \, d\Imm \zeta,\\
B &= \frac{1}{2\pi i}\int_{|\zeta|=1}\bigg\{\partial_z K(\zeta)\left[-1+ \frac{1}{\pi}\int_{\C}r(\lambda,z)\overline{\Omega_2(\lambda,\zeta)}d\Ree \lambda \, d \Imm \lambda\right] \\ \nonumber
&\qquad +\frac{K(\zeta)}{\pi} \int_{\C}\bigg[ r(\lambda,z)\bigg(-\frac{i}{2}\sqrt{E}\left(\frac{1}{\lambda}+\bar \lambda\right)\overline{\Omega_2(\lambda,\zeta)} \\ \nonumber
&\qquad +\overline{\partial_{\bar z}\Omega_2(\lambda,\zeta)}\bigg)\bigg]d\Ree \lambda \, d\Imm \lambda \bigg\}d \zeta, \\
C &= -\frac{1}{2\pi i}\int_{|\zeta|=1}\bigg\{\overline{\partial_{\bar z} K(\zeta)}\left[\frac{1}{\pi}\int_{\C}r(\lambda,z)\overline{\Omega_1(\lambda,\zeta)}d\Ree \lambda \, d \Imm \lambda\right] \\ \nonumber
&\qquad +\frac{\overline{K(\zeta)}}{\pi} \int_{\C}\bigg[ r(\lambda,z)\bigg(-\frac{i}{2}\sqrt{E}\left(\frac{1}{\lambda}+\bar \lambda\right)\overline{\Omega_1(\lambda,\zeta)} \\ \nonumber
&\qquad +\overline{\partial_{\bar z}\Omega_1(\lambda,\zeta)}\bigg)\bigg]d\Ree \lambda \, d\Imm \lambda \bigg\}d \bar \zeta.
\end{align}
Now let $v_1, v_2$ be two potential satisfying the assumptions of Theorem \ref{maintheo}. Let $\mu^j_{-1}, r_j, \rho_j,e_j, K_j, \Omega^j_1,\Omega^j_2,X^j_1,X^j_2, A_j,B_j,C_j$ the above-defined functions corresponding to $v_j$, for $j=1,2$. Then
\begin{align}\label{formv}
 v_2(z)-v_1(z) = 2i\sqrt{E}(A_2-A_1 + B_2 - B_1+ C_2-C_1).
\end{align}
In order to estimate $A_2-A_1$, $B_2 - B_1$ and $C_2-C_1$ we will need the following two propositions.

\begin{prop} \label{propK}
Let $v_j$, $j=1,2$, be two potential satisfying the assumptions of Theorem \ref{maintheo}. Then we have, for every $a >1$, $E\geq E_3(N,m)$ and $1<s<2<s'<+\infty$,
\begin{align}
\|K_2-K_1\|_{L^2(T)} &\leq c(N,m)\bigg[\|\rho_2-\rho_1\|_{L^2(T^2)}+\frac{\|r_2-r_1\|_{L^{s,s'}(\C)}}{\sqrt{E}} \\ \nonumber
&\quad + \frac{1}{E}\left(\delta r_a+\frac{a}{a-1}\|r_2-r_1\|_{L^{s,s'}(\C \setminus D_a)}\right) \bigg],
\end{align}
where $\|\cdot\|_{L^{s,s'}} = \|\cdot\|_{L^{s}}+\|\cdot\|_{L^{s'}}$, $T$ is the unit circle and
\begin{align} \label{defda}
D_a = \{\lambda \in \C \, | \, 1/a < |\lambda| < a\},\\
\delta r_a = \sup_{\lambda \in D_a}|r_2(\lambda,z)-r_1(\lambda,z)|.
\end{align}
\end{prop}

\begin{prop}\label{propKK}
Let $v_j$, $j=1,2$, be two potential satisfying the assumptions of Theorem \ref{maintheo}. Then we have, for every $a >1$, $E\geq E_4(N,m)$ and for some $1<s<2<s'<+\infty$,
\begin{align}
&\|\nabla K_2- \nabla K_1\|_{L^2(T)} \leq c(N,m)\Bigg[E^{1/2}\|\rho_2-\rho_1\|_{L^2(T^2)}\\ \nonumber
&\quad +\left\|\left(\frac{1}{|\lambda|}+|\lambda|\right)(r_2 - r_1)\right\|_{L^{s,s'}(\C)}\\ \nonumber
&\quad + E^{-1}\left(\delta r'_a + \frac{a}{a-1}\left\|\left(\frac{1}{|\lambda|}+|\lambda|\right)(r_2 - r_1)\right\|_{L^{s,s'}(\C \setminus D_a)}\right) \Bigg]
\end{align}
where $\nabla$ is taken with respect to $z$, $\|\cdot\|_{L^{s,s'}} = \|\cdot\|_{L^{s}}+\|\cdot\|_{L^{s'}}$, $D_a$ is defined in \eqref{defda} and
\begin{align}
\delta r'_a = \sup_{D_a} \left( \frac{1}{|\lambda|}+|\lambda|\right)|(r_2 - r_1)(\lambda)|.
\end{align}
\end{prop}

\begin{proof}[Proof of Proposition \ref{propK}]
We rewrite integral equation \eqref{defK} for $K_j$ as follows:
\begin{equation} \label{eqK}
 (I-\Theta_j)K_j(\lambda)=\int_{|\lambda'|=1}\rho_j(\lambda,\lambda',z)e_j(\lambda')|d\lambda'|,
\end{equation}
where
\begin{align}
 \Theta_j f (\lambda) &= \frac{1}{2\pi i}\int_{|\lambda'|=1}\rho_j(\lambda,\lambda',z)\bigg[\int_{|\zeta|=1}\Omega^j_1(\lambda'(1+0),\zeta)f(\zeta)d\zeta\\ \nonumber
&\quad +\Omega^j_2(\lambda',\zeta)\overline{f(\zeta)}d\bar\zeta\bigg]|d\lambda'|, \quad j=1,2,
\end{align}
for $f\in L^2(T)$. Subtracting equation \eqref{eqK} for $j=2$ and $j=1$ gives
\begin{align} \label{eqkdiff}
&(I-\Theta_2)(K_2-K_1)(\lambda) = (\Theta_2-\Theta_1)K_1(\lambda)\\ \nonumber
&\quad+ \int_{|\lambda'|=1}(\rho_2-\rho_1)(\lambda,\lambda',z)e_2(\lambda')+\rho_1(\lambda, \lambda',z)(e_2-e_1)(\lambda')|d\lambda'|.
\end{align}
We will now use some results of \cite{N3} and \cite[Ch. III]{V}. Let $L^p_{\nu}(\C)$ the set
defined in \eqref{deflpnu}. From Lemma \ref{lemb} we have that $r_j \in L^p_{\nu}(\C)$ for all $\nu < m$. In particular $\|r_j\|_{L^p_{2}(\C)}\leq c(N,p)E^{-m/2}$, for $m >2$, $p \geq 1$.
Then we following estimates holds
\begin{align} \label{estomega1}
\left|\Omega_1^j(\lambda,\zeta)-\frac{1}{\zeta-\lambda}\right| &\leq c(N,m,p)E^{-m/2}\frac{1}{|\zeta-\lambda|^{2/p}},\\ \label{estomega2}
\left|\Omega_2^j(\lambda,\zeta)\right| &\leq c(N,m,p)E^{-m/2}\frac{1}{|\zeta-\lambda|^{2/p}},
\end{align}
since $\|r_j\|_{L^p_{2}(\C)} \leq c(N,m,p)E^{-m/2}$, $j=1,2$. These estimates are proved in \cite[Ch. III, \S 8]{V}. We also recall the following classical inequality:
\begin{align}
&\|C_{\pm}u\|_{L^p(T)} \leq c(p)\|u\|_{L^p(T)}, \quad 1 < p < +\infty,\\
&(C_{\pm}u)(\lambda) = \frac{1}{2\pi i}\int_T \frac{u(\zeta)}{\zeta-\lambda(1\mp 0)}d\zeta.
\end{align}
Then we have, for $p >4$,
\begin{align*}
\|\Theta_j f\|_{L^2(T)} &\leq c(r_0,p) \sup_{\lambda \in T}\|\rho_j(\lambda,\cdot,z)\|_{L^1(T)}\left(1+\|1/|\cdot - \lambda|^{2/p}\|_ {L^2(\C)}\right) \|f\|_{L^2(T)}\\
& \leq \frac{c(r_o,p)}{\sqrt{E}}\|f\|_{L^2(T)},
\end{align*}
where we used the fact that $\rho_j$ satisfies inequality \eqref{decrh} and
\begin{align}\label{bella}
\int_T(1+E|\lambda-\lambda'|^2)^{-m/2}|d\lambda'| \leq c E^{-1/2}.
\end{align}
Then for $E \geq E_4(r_0,m)$ we can solve equation \eqref{eqkdiff} by iteration in $L^2(T)$ and find
\begin{align} \label{estkkk}
\|K_2-K_1\|_{L^2(T)} &\leq c(N,m)\bigg( \|(\Theta_2-\Theta_1)K_1\|_{L^2(T)}\\ \nonumber
&\quad +\|\rho_2-\rho_1\|_{L^2(T^2)}+\frac{\|e_2-e_1\|_{L^{\infty}(T)}}{\sqrt{E}}\bigg),
\end{align}
where we used the $L^{\infty}(T)$-boundedness of $e_2$ (which follows from considerations at the beginning of this section). We have that 
\begin{align}
\|e_2-e_1\|_{L^{\infty}(T)} \leq c(N,m)\left(\|r_2-r_1\|_{L^s(\C)} +\|r_2-r_1\|_{L^{s'}(\C)}\right),
\end{align}
for $1<s<2<s'<+\infty$. Indeed, this follows from the integral equation
\begin{align*}
e_2(\lambda)-e_1(\lambda) = -\frac{1}{\pi}\int_{\C}\frac{(r_2-r_1)(\zeta,z)\overline{e_2(\zeta)}}{\zeta-\lambda}+\frac{r_2(\zeta,z)\overline{(e_2-e_1)(\zeta)}}{\zeta-\lambda}d\Ree \zeta\, d\Imm \zeta,
\end{align*}
which is a consequence of \eqref{defe}, from H\"older's inequality, Lemma \ref{lemb}, estimate \eqref{estmula1} and the $L^{\infty}$-boundedness of $e_j(\lambda)$ (see the beginning of this section and Section \ref{secpf} for more details).

The first term of the right hand side of \eqref{estkkk}, $(\Theta_2-\Theta_1)K_1$, can be written as
\begin{align} \nonumber
(\Theta_2-\Theta_1)K_1&=\frac{1}{2\pi i}\int_T (\rho_2-\rho_1)\int_{T}\Omega_1^2 K_1 d\zeta+\Omega_2^2\overline{K_1}d\bar \zeta |d\lambda'|\\ \nonumber
&\quad + \frac{1}{2\pi i}\int_T \rho_1\int_{T}(\Omega_1^2-\Omega_1^1) K_1 d\zeta+(\Omega_2^2-\Omega_2^1)\overline{K_1}d\bar \zeta |d\lambda'|,
\end{align}
where we dropped the dependence on every variable for simplicity's sake. We obtain
\begin{align}
&\|(\Theta_2-\Theta_1)K_1\|_{L^2(T)}\leq c(N,m)E^{-1/2}\Big(\|\rho_2-\rho_1\|_{L^2(T^2)} \\ \nonumber
&\quad + E^{-1/2}\sup_{\lambda \in T}\left(\|(\Omega_1^2-\Omega_1^1)(\lambda,\cdot)\|_{L^{2}(T)}+ \|(\Omega_2^2-\Omega_2^1)(\lambda,\cdot)\|_{L^{2}(T)}\right) \Big),
\end{align}
where we used the fact that 
\begin{equation} \label{estK1}
\|K_j\|_{L^{\infty}(T)}\leq c(N,m)E^{-1/2},
\end{equation}
which follows from the first equality in \eqref{defK}, equation \eqref{corrt}, inequality \eqref{decrh} for $\rho_j$, Lemma \ref{lem21} and estimate \eqref{bella}.

We now need to estimate the difference of $\Omega^j_k$. The $X^j_k$ satisfy
\begin{align} \label{dbarx21}
\partial_{\bar \lambda}\left(X^2_k-X^1_k\right)(\lambda,\zeta)&=r_1(\lambda,z)\overline{(X^2_k-X^1_k)(\lambda,\zeta)}\\ \nonumber
&\quad +(r_2-r_1)(\lambda,z)\overline{X^2_k}(\lambda,\zeta),
\end{align}
for $k=1,2$. Note that the last equation holds over all the complex plane, since $X^2_k-X^1_k$ has no singularity. Moreover $(X^2_k-X^1_k)(\cdot,\zeta) \in L^p(\C)$, for every $p>2$, thanks to properties of the integral operator in \eqref{defX1}, \eqref{defX2}, summarized in Section \ref{secpf} (see \eqref{estdbarr} for instance). Then we may define, following \cite{V},
\begin{align}
w_k(\lambda,\zeta) = \partial_{\bar \lambda}^{-1} \left(r_1(\cdot,z)\frac{\overline{(X^2_k-X^1_k)}(\cdot,\zeta)}{(X^2_k-X^1_k)(\cdot,\zeta)}\right)(\lambda),\\ \label{defdbarr}
\partial_{\bar \lambda}^{-1}f(\lambda) = -\frac{1}{\pi}\int_{\C}\frac{f(\eta)}{\eta-\lambda}d\Ree \eta\,d\Imm \eta.
\end{align}
We have that $\|w_k\|_{L^{\infty}(\C)}\leq c\left(\|r_1\|_{L^s(\C)}+\|r_1\|_{L^{s'}(\C)}\right)$, for any fixed $1<s<2<s'<+\infty$. Thus we have the following representation formula
\begin{align}
X^2_k-X^1_k=e^{w_k}\partial_{\bar \lambda}^{-1}\left(e^{-w_k}(r_2-r_1)\overline{X^2_k}\right),
\end{align}
which yields
\begin{align} \label{difxk}
&|(X^2_k-X^1_k)(\lambda,\zeta)|\leq c(N,m)\int_{\C}\frac{|(r_2-r_1)(\eta,z)||X^2_k(\eta,\zeta)|}{|\eta -\lambda|}d\Ree \eta \, d\Imm \eta\\ \nonumber
&\quad \leq c(N,m)\int_{\C}\left(\frac{|(r_2-r_1)(\eta,z)|}{|\eta -\lambda||\zeta-\eta|}+\frac{|(r_2-r_1)(\eta,z)|}{|\eta -\lambda||\zeta-\eta|^{2/p}}\right)d\Ree \eta \, d\Imm \eta,
\end{align}
for $\lambda,\zeta \in T$, where we used estimates \eqref{estomega1} and \eqref{estomega2}. Now take $a >1$ and define
\begin{align}
D_a = \{\lambda \in \C \, | \, 1/a < |\lambda| < a\},\\
\delta r_a = \sup_{\lambda \in D_a}|r_2(\lambda,z)-r_1(\lambda,z)|.
\end{align}
We then estimate the last integral in \eqref{difxk} on $D_a$ and on $\C \setminus D_a$, like in \cite[Theorem 6.1]{N3}, and we obtain, using H\"older's inequality,
\begin{align}
&|(X^2_k-X^1_k)(\lambda,\zeta)|\leq c(N,m)\left(\delta r_a |\log|\lambda-\zeta||+\frac{a}{a-1}\|r_2-r_1\|_{L^{s,s'}(\C \setminus D_a) }\right),
\end{align}
for $\lambda, \zeta \in T$, $1<s<2<s'<+\infty$, since $\frac{a}{a-1}= \max(\frac{1}{a-1}, \frac{1}{1-1/a})$. We then get
\begin{equation}
\sup_{\lambda \in T}\|\Omega_k^2(\lambda,\cdot)-\Omega_k^1(\lambda,\cdot)\|_{L^2(T)}\leq c(N,m)\left(\delta r_a+\frac{a}{a-1}\|r_2-r_1\|_{L^{s,s'}(\C \setminus D_a)}\right).
\end{equation}
First this yields 
\begin{align} \label{esttheta}
&\|(\Theta_2-\Theta_1)K_1\|_{L^2(T)}\leq c(N,m)E^{-1/2}\bigg[\|\rho_2-\rho_1\|_{L^2(T^2)} \\ \nonumber
&\quad + E^{-1/2}\left(\delta r_a+\frac{a}{a-1}\|r_2-r_1\|_{L^{s,s'}(\C \setminus D_a)}\right) \bigg],
\end{align}
and finally
\begin{align*}
\|K_2-K_1\|_{L^2(T)} &\leq c(N,m)\bigg[\|\rho_2-\rho_1\|_{L^2(T^2)}+\frac{\|r_2-r_1\|_{L^{s,s'}(\C)}}{\sqrt{E}} \\ \nonumber
&\quad + \frac{1}{E}\left(\delta r_a+\frac{a}{a-1}\|r_2-r_1\|_{L^{s,s'}(\C \setminus D_a)}\right) \bigg],
\end{align*}
which ends the proof of Proposition \ref{propK}.
\end{proof}

\begin{proof}[Proof of Proposition \ref{propKK}]
We derive integral equation \eqref{defK} for $K_j$ with respect to $\partial_z$ and $\partial_{\bar z}$ and we obtain two coupled integral equations for $\partial_z K_j$ and $\partial_{\bar z}K_j$. Thus we define 
\begin{equation}
K^{\pm}_j= \partial_z K_j \pm \partial_{\bar z}K_j, \quad j=1,2,
\end{equation}
which satisfy
\begin{align} \label{intderK}
&(I-\Theta^{\pm}_j)K^{\pm}_j(\lambda)=\int_{|\lambda'|=1}\rho^{\pm}_j(\lambda,\lambda',z)\Big[e_j(\lambda')\\ \nonumber
&\quad +\frac{1}{2\pi i}\int_{|\zeta|=1}\Omega^j_1(\lambda'(1+0),\zeta)K_j(\zeta)d\zeta +\Omega^j_2(\lambda',\zeta)\overline{K_j(\zeta)}d\bar\zeta\Big]|d\lambda'|\\ \nonumber
&\quad \int_{|\lambda'|=1}\rho_j(\lambda,\lambda',z)\Big[e^{\pm}_j(\lambda')+\frac{1}{2\pi i}\int_{|\zeta|=1}\Omega^{\pm,j}_1(\lambda'(1+0),\zeta)K_j(\zeta)d\zeta\\ \nonumber &\quad +\Omega^{\pm,j}_2(\lambda',\zeta)\overline{K_j(\zeta)}d\bar\zeta\Big]|d\lambda'|, \quad j=1,2,
\end{align}
where
\begin{align}
 \Theta^{\pm}_j f (\lambda) &= \frac{1}{2\pi i}\int_{|\lambda'|=1}\rho_j(\lambda,\lambda',z)\bigg[\int_{|\zeta|=1}\Omega^j_1(\lambda'(1+0),\zeta)f(\zeta)d\zeta\\ \nonumber
&\quad \pm \Omega^j_2(\lambda',\zeta)\overline{f(\zeta)}d\bar\zeta\bigg]|d\lambda'|,\\ 
\rho^{\pm}_j &= \partial_z \rho_j \pm \partial_{\bar z}\rho_j,  \qquad e^{\pm}_j = \partial_z e_j \pm \partial_{\bar z} e_j\\ 
\Omega^{\pm,j}_k &=  \partial_z \Omega^j_k \pm \partial_{\bar z}\Omega^j_k,
\end{align}
for $j=1,2$, $k=1,2$. Integral equations \eqref{intderK} are obtained by adding and subtracting the two above-mentioned coupled integral equations for $\partial_z K_j$ and $\partial_{\bar z}K_j$ (which are obtained from \eqref{defK}).

We then subtract \eqref{intderK} for $j=1$ from $j=2$ and we get
\begin{align*}
&(I-\Theta^{\pm}_2)(K^{\pm}_2-K^{\pm}_1)=(\Theta^{\pm}_2-\Theta^{\pm}_1)K^{\pm}_1\\ \nonumber
&\quad + \int_{|\lambda'|=1}(\rho^{\pm}_2-\rho^{\pm}_1)\Big[e_2 +\frac{1}{2\pi i}\int_{|\zeta|=1}\Omega^2_1 K_2 d\zeta +\Omega^2_2\overline{K_2}d\bar\zeta\Big]|d\lambda'|\\ \nonumber
&\quad + \int_{|\lambda'|=1}\rho^{\pm}_1\Big[e_2-e_1+\frac{1}{2 \pi i}\int_{|\zeta|=1}\left( (\Omega^2_1-\Omega^1_1) K_2 + \Omega^1_1(K_2-K_1)\right) d\zeta \\ \nonumber
&\quad +\left( (\Omega^2_2-\Omega^1_2)\bar{K_2}+\Omega^1_2 (\bar K_2 - \bar K_1)\right) d\bar\zeta\Big]|d\lambda'|\\
&\quad + \int_{|\lambda'|=1}(\rho_2-\rho_1)\Big[e^{\pm}_2 +\frac{1}{2\pi i}\int_{|\zeta|=1}\Omega^{\pm,2}_1 K_2 d\zeta +\Omega^{\pm,2}_2\bar{K_2}d\bar\zeta\Big]|d\lambda'| \\
&\quad + \int_{|\lambda'|=1}\rho_1\Big[e^{\pm}_2-e^{\pm}_1+\frac{1}{2 \pi i}\int_{|\zeta|=1}\left( (\Omega^{\pm,2}_1-\Omega^{\pm,1}_1) K_2 + \Omega^{\pm,1}_1(K_2-K_1)\right) d\zeta \\ \nonumber
&\quad +\left( (\Omega^{\pm,2}_2-\Omega^{\pm,1}_2)\bar{K_2}+\Omega^{\pm,1}_2 (\bar K_2 - \bar K_1)\right) d\bar\zeta\Big]|d\lambda'|,
\end{align*}
where we dropped the dependence on every variable for simplicity's sake. The operator $\Theta^{\pm}_j$ satisfies the same estimates of operator $\Theta_j$. Then, for $E \geq E_4(N,m)$ the last equation is solvable by iteration in $L^2(T)$ and we have
\begin{align*}
&\|K^{\pm}_2-K^{\pm}_1\|_{L^2(T)}\leq c(N)\Bigg( \|(\Theta^{\pm}_2-\Theta^{\pm}_1)K^{\pm}_1\|_{L^2(T)}\\ \nonumber
&\quad + \|\rho^{\pm}_2-\rho^{\pm}_1\|_{L^2(T^2)} + \sup_{\lambda \in T}\| \rho^{\pm}_1(\lambda,\cdot) \|_{L^1(T)} \Big[\|e_2-e_1\|_{L^{\infty}(T)}\\ \nonumber
&\quad +\|K_2-K_1\|_{L^2(T)} + \frac{1}{\sqrt{E}}( \sup_{\lambda \in T}\|(\Omega^2_1-\Omega^1_1)(\lambda,\cdot)\|_{L^2(T)}\\ 
&\quad +\sup_{\lambda \in T}\|(\Omega^2_2-\Omega^1_2)(\lambda,\cdot)\|_{L^2(T)} )\Big]\\
&\quad + \|\rho_2-\rho_1\|_{L^2(T^2)}\left[\|e^{\pm}_2\|_{L^{\infty}(T)} +\left\|\int_{|\zeta|=1}\Omega^{\pm,2}_1 K_2 d\zeta +\Omega^{\pm,2}_2\bar{K_2}d\bar\zeta \right\|_{L^2(T)}\right] \\
&\quad + \sup_{\lambda \in T}\|\rho_1(\lambda,\cdot)\|_{L^1(T)}\Bigg[\|e^{\pm}_2-e^{\pm}_1\|_{L^{\infty}(T)}\\ \nonumber
&\quad +\frac{1}{2 \pi i}\bigg\| \int_{|\zeta|=1}\left( (\Omega^{\pm,2}_1-\Omega^{\pm,1}_1) K_2 + \Omega^{\pm,1}_1(K_2-K_1)\right) d\zeta \\ \nonumber
&\quad +\left( (\Omega^{\pm,2}_2-\Omega^{\pm,1}_2)\bar{K_2}+\Omega^{\pm,1}_2 (\bar K_2 - \bar K_1)\right) d\bar\zeta \bigg\|_{L^2(T)} \Bigg] \Bigg),
\end{align*}
where we used the $L^{\infty}$-boundedness of $e_j$ (see the beginning of this section) and estimates \eqref{estomega1}, \eqref{estomega2} as in the proof of Proposition \ref{propK}.

Since the kernels of $\Theta_j$ and $\Theta^{\pm}_j$ differ only by a sign, estimate \eqref{esttheta} yields
\begin{align}
&\|(\Theta^{\pm}_2-\Theta^{\pm}_1)K^{\pm}_1\|_{L^2(T)} \leq c(N,m)E^{-1/2}\bigg[\|\rho_2-\rho_1\|_{L^2(T^2)} \\ \nonumber
&\quad + E^{-1/2}\left(\delta r_a+\frac{a}{a-1}\|r_2-r_1\|_{L^{s,s'}(\C \setminus D_a)}\right) \bigg],
\end{align}
where we used the fact that 
\begin{equation} \label{estkpm}
\|K^{\pm}_j\|_{L^{\infty}(T)}\leq c(N,m) E^{-1/2}.
\end{equation}
Indeed this follows from the first equality in \eqref{defK}, equation \eqref{corrt} (derived with respect to $\partial_z$ and $\partial_{\bar z}$), inequality \eqref{decrh} for $\rho_j$, Lemma \ref{lem21} and estimate
\begin{align}
\int_T \sqrt{E}|\lambda-\lambda'|(1+E|\lambda-\lambda'|^2)^{-m/2}|d\lambda'| \leq c E^{-1/2}.
\end{align}
This inequality, as well as \eqref{decrh}, also implies
\begin{equation}
\sup_{\lambda \in T}\| \rho^{\pm}_1(\lambda,\cdot) \|_{L^1(T)} \leq c(N,m) E^{-1/2}.
\end{equation}
We also have
\begin{align}
\|\rho^{\pm}_2-\rho^{\pm}_1\|_{L^2(T^2)} \leq c \sqrt{E}\|\rho_2-\rho_1\|_{L^2(T^2)},
\end{align}
which follows from the definition of $\rho(\lambda,\lambda',z)$ in \eqref{defrhoz}, and
\begin{equation}
 \sup_{\lambda \in T}\| \rho_1(\lambda,\cdot) \|_{L^1(T)} \leq c(N,m) E^{-1/2}.
\end{equation}
In order to estimate the difference of $e^{\pm}_j$ we proceed as follows. From \eqref{defe} we have that $e^{\pm}_j$ satisfies the following integral equation
\begin{align}
e^{\pm}_j(\lambda) = -\frac{1}{\pi}\int_{\C}\left(\frac{r_j^{\pm}(\zeta,z)\bar e_j(\zeta)}{\zeta-\lambda}\pm \frac{r_j(\zeta,z)\bar e^{\pm}_j(\zeta)}{\zeta-\lambda} \right)d\Ree \zeta \, d\Imm \zeta,
\end{align}
which gives
\begin{align}
 &(e^{\pm}_2-e^{\pm}_1)(\lambda) = -\frac{1}{\pi}\int_{\C}\bigg(\frac{(r_2^{\pm}-r_1^{\pm})(\zeta,z)\bar e_2(\zeta)+r_1^{\pm}(\zeta,z)(\bar e_2-\bar e_1)(\zeta)}{\zeta-\lambda} \\ \nonumber
&\quad \pm \frac{(r_2-r_1)(\zeta,z)\bar e^{\pm}_2(\zeta) +r_1(\zeta,z)(\bar e^{\pm}_2-\bar e^{\pm}_1)(\zeta)}{\zeta-\lambda} \bigg) d\Ree \zeta \, d\Imm \zeta.
\end{align}
Using several times H\"older's inequality as well as estimates \eqref{estmula1}, \eqref{estmula2} (see Section \ref{secpf}), Lemma \ref{lemestrr}, definition \eqref{defr2} and the $L^{\infty}$ boundedness of $e_j, e^{\pm}_j$, we obtain
\begin{align}
 \|e_2^{\pm}-e_1^{\pm}\|_{L^{\infty}(T)} &\leq c(N,m)\bigg(\|r_2-r_1\|_{L^{s,s'}(\C)} \\ \nonumber 
&\quad + \sqrt{E}\left\|\left(\frac{1}{|\lambda|}+|\lambda|\right)|r_2-r_1|\right\|_{L^{s,s'}(\C)}\bigg),
\end{align}
for $1<s<2<s'< \infty$. 

We now pass to the estimates of the $\Omega^{\pm,j}_k$. Define
\begin{align} \label{defxpm}
X^{\pm,j}_k &=  \partial_z X^j_k \pm \partial_{\bar z}X^j_k, \qquad j=1,2, \quad k=1,2.
\end{align}
From definitions \eqref{defX1}, \eqref{defX2} we have that $X^{\pm,j}_k$ satisfy the following non-homogeneous $\bar \partial$ equations
\begin{align} \label{eqxpm}
\partial_{\bar \lambda}X^{\pm,j}_k = \pm r_j \overline{X^{\pm,j}_k}+r_j^{\pm}\overline{X^j_k},
\end{align}
where $X^{-,j}_k$ has no singularities while $X^{+,j}_k$ has a pole at $\lambda = \zeta$. More precisely
\begin{align}
\lim_{\lambda \to \zeta}(\lambda-\zeta)X^{+,j}_1(\lambda,\zeta) = 1, \qquad \lim_{\lambda \to \zeta}(\lambda-\zeta)X^{+,j}_2(\lambda,\zeta) = \frac 1 i.
\end{align}
We will now estimate the $X^{\pm,j}_k$ using an argument of Vekua \cite[Ch. III, \S 7-8]{V}. Consider the following inverse of $\partial_{\bar \lambda}$:
\begin{align}
\partial_{\bar \lambda}^{-1}f(\lambda, \zeta) = -\frac{\zeta-\lambda}{\pi}\int_{\C}\frac{f(\eta)}{(\eta-\lambda)(\zeta-\eta)}d\Ree \eta\,d\Imm \eta,
\end{align}
defined for $f \in L^p_2(\C)$. It satisfies the following inequalities
\begin{align} \label{estdbar1}
|\partial_{\bar \lambda}^{-1}f(\lambda, \zeta)|&\leq c(p)\|f\|_{L^p_2(\C)},\\ \label{estdbar2}
|\partial_{\bar \lambda}^{-1}f(\lambda, \zeta)|&\leq c(p)\|f\|_{L^p_2(\C)}|\lambda-\zeta|^{1-2/p},
\end{align}
which are proved in \cite[Ch. III, \S 4]{V}. Let
\begin{align}
w^{\pm,j}_k(\lambda,\zeta)= \partial_{\bar \lambda}^{-1}\left(\pm r_j(\cdot,z)\frac{\overline{ X^{\pm,j}_k}(\cdot,\zeta)}{ X^{\pm,j}_k(\cdot,\zeta)} \right)(\lambda,\zeta).
\end{align}
We first consider $X^{-,j}_k$. Since it has no singularity, we can argue as in the proof of Proposition \ref{propK}, and find the representation formula
\begin{align}
X^{-,j}_k(\lambda,\zeta) = e^{w^{-,j}_k(\lambda,\zeta)}\partial_{\bar \lambda}^{-1}\left(e^{-w^{-,j}_k(\cdot,\zeta)}\left(r^-_j(\cdot,z) \overline{X^j_k}(\cdot,\zeta)\right) \right)(\lambda,\zeta),
\end{align}
which yields
\begin{equation} \label{estx-}
|X^{-,j}_k(\lambda,\zeta)|\leq c(N,m)E^{-(m-1)/2},
\end{equation}
for $\lambda \in \C$, $\zeta \in T$. Indeed this follows from \eqref{estdbar1}, the boundedness of $\|r_j\|_{L^p_2(\C)}$ and the fact that
\begin{align} \label{estrpmp}
\|r^{\pm}(\cdot,z)\|_{L^p_2(\C)} &\leq c\sqrt{E}\|(1+|\lambda|^{2-2/p})r(\cdot,z)\|_{L^p(\C)} \\ \nonumber
&\leq  c(N,m)E^{-(m-1)/2},
\end{align}
for $p \geq 1$. We also used H\"older inequality and the estimates
\begin{align} \label{estx1}
\left|X_1^j(\lambda,\zeta)-\frac{1}{2(\zeta-\lambda)}\right| &\leq c(N,p)\frac{1}{|\zeta-\lambda|^{2/p}},\\ \label{estx2}
\left|X_2^j(\lambda,\zeta)-\frac{1}{2i(\zeta-\lambda)}\right| &\leq c(N,p)\frac{1}{|\zeta-\lambda|^{2/p}},
\end{align}
which follow from \eqref{estomega1} and \eqref{estomega2}.

For $X^{+,j}_k$ the two representation formulas hold:
\begin{align}\label{reprX+1}
X^{+,j}_1(\lambda,\zeta) &=\frac{ e^{w^{+,j}_1(\lambda,\zeta)}\partial_{\bar \lambda}^{-1}\left(e^{-w^{+,j}_1(\cdot,\zeta)}\left(r^+_j(\cdot,z) (\zeta-\cdot)\overline{X^j_1}(\cdot,\zeta)\right) \right)(\lambda,\zeta)}{\zeta-\lambda},\\ \label{reprX+2}
X^{+,j}_2(\lambda,\zeta) &= \frac{e^{w^{+,j}_2(\lambda,\zeta)}\partial_{\bar \lambda}^{-1}\left(e^{-w^{+,j}_2(\cdot,\zeta)}\left(r^+_j(\cdot,z) i(\zeta-\cdot)\overline{X^j_2}(\cdot,\zeta)\right) \right)(\lambda,\zeta)}{i(\zeta-\lambda)}.
\end{align}
These formulas (non-linear integral equations) are some sort of generalisations of the non-linear integral equation (7.3) in \cite[Ch. III]{V} and may be generalised to solutions of non-homogeneous $\bar \partial$ equations with arbitrary prescribed (analytic) singularities.

To prove \eqref{reprX+1} we proceed as follows. We define $ X' (\lambda,\zeta) = (\zeta-\lambda)X^{+,j}_1(\lambda,\zeta)$, which is continuous and satisfies
\begin{align}
\partial_{\bar \lambda}X' (\lambda,\zeta)= r_j\frac{\zeta-\lambda}{\bar \zeta - \bar \lambda} \overline{X' (\lambda,\zeta)}+r^+_j(\lambda,z)(\zeta-\lambda)\overline{X^j_1}(\lambda,\zeta).
\end{align}
Then we have
\begin{equation}
\partial_{\bar \lambda}(e^{-w^{+,j}_1(\lambda,\zeta)}X' (\lambda,\zeta)) = e^{-w^{+,j}_1(\lambda,\zeta)} r^+_j(\lambda,z)(\zeta-\lambda)\overline{X^j_1}(\lambda,\zeta).
\end{equation}
It is then possible to apply $\partial_{\bar \lambda}$ since we have estimates \eqref{estrpmp}-\eqref{estx2}, which guarantees that the right hand side is in $L^p_2(\C)$, for $\zeta \in T$. The proof of \eqref{reprX+2} is completely analogous.

From \eqref{reprX+1}, \eqref{reprX+2}, as well as \eqref{estrpmp}-\eqref{estx2} and \eqref{estdbar2} we find
\begin{align}\label{estx+}
|X^{+,j}_k(\lambda,\zeta)| \leq c(N,m)E^{-(m-1)/2}|\zeta -\lambda|^{-2/p},
\end{align}
for $\lambda \in \C$, $\zeta \in T$, $j=1,2$, $k=1,2$, $p >2$. To summarize, we have obtained:
\begin{align} \label{estomega+}
|\Omega^{-,j}_k(\lambda,\zeta)|&\leq c(N,m)E^{-(m-1)/2}, \\ \label{estomega-}
|\Omega^{+,j}_k(\lambda,\zeta)|&\leq c(N,m)E^{-(m-1)/2}|\zeta -\lambda|^{-2/p},
\end{align}
for $\lambda \in \C$, $\zeta \in T$, $j=1,2$, $k=1,2$, $p >2$.

We can now estimate the difference $\Omega^{\pm,2}_k-\Omega^{\pm,1}_k$ using similar arguments. The functions $X^{\pm,2}_k-X^{\pm,1}_k$ are continuous and satisfy
\begin{align} \label{dbarpm}
&\partial_{\bar \lambda}(X^{\pm,2}_k-X^{\pm,1}_k)= \pm r_1(\bar X^{\pm,2}_k- \bar X^{\pm,1}_k)\\ \nonumber
&\quad +(r^{\pm}_2-r^{\pm}_1)\bar X^2_k+r_1^{\pm}(\bar X^{2}_k- \bar X^{1}_k)\pm(r_2-r_1)\bar X^{\pm,2}_k,
\end{align}
for $\lambda , \zeta \in \C$, $j=1,2$, $k=1,2$. Using the $\partial_{\bar \lambda}^{-1}$ defined in \eqref{defdbarr} and arguing as above we find
\begin{align}
|(X^{\pm,2}_k-X^{\pm,1}_k)(\lambda,\zeta)| \leq c(N)(J_1(\lambda,\zeta)+J_2(\lambda,\zeta) + J_3(\lambda,\zeta)),
\end{align}
where
\begin{align}
J_1(\lambda,\zeta) &= \int_{\C}\frac{|(r^{\pm}_2-r^{\pm}_1)(\eta,z)||X^2_k(\eta,\zeta)|}{|\eta -\lambda|}d\Ree \eta \, d\Imm \eta,\\ 
J_2(\lambda,\zeta) &= \int_{\C}\frac{|r^{\pm}_1(\eta,z)||(X^2_k-X^1_k)(\eta,\zeta)|}{|\eta -\lambda|}d\Ree \eta \, d\Imm \eta,\\ 
J_3(\lambda,\zeta) &= \int_{\C}\frac{|(r_2-r_1)(\eta,z)||X^{\pm,2}_k(\eta,\zeta)|}{|\eta -\lambda|}d\Ree \eta \, d\Imm \eta.
\end{align}
For $J_1$ and $J_3$ we argue as in the proof of Proposition \ref{propK}. We find, for $a >1$, $p >2$,
\begin{align} \label{estj1}
J_1(\lambda,\zeta) &\leq c(N,m)\Bigg(\delta r'_a |\log|\zeta-\lambda|| \\ \nonumber
&\quad + \frac{a}{a-1}\left\|\left(\frac{1}{|\lambda|}+|\lambda|\right)(r_2 - r_1)\right\|_{L^{s,s'}(\C \setminus D_a)}\Bigg),
\end{align}
for $\lambda, \zeta \in T$, where $\delta r'_a = \sup_{D_a} \left( \frac{1}{|\lambda|}+|\lambda|\right)|(r_2 - r_1)(\lambda)|$, $1<s<2<s'<+\infty$. Using \eqref{estx+} for $p >4$ and \eqref{estx-}, we obtain for both "$\pm$" cases
\begin{align} \label{estj3}
J_3(\lambda,\zeta) \leq c(N,m)E^{-(m-1)/2}\left(\delta r_a + \frac{a}{a-1}\left\|r_2 - r_1\right\|_{L^{s,s'}(\C \setminus D_a)}\right),
\end{align}
for $\lambda, \zeta \in T$.

In order to estimate $J_2$ we start with H\"older's inequality for $q >2$, $1/p+1/q =1$:
\begin{align}
J_2(\lambda,\zeta) &\leq \left\|\frac{r^{\pm}_1(\cdot,z)}{|\cdot-\lambda|}\right\|_{L^p(\C)} \|(X^2_k-X^1_k)(\cdot,\zeta)\|_{L^q(\C)}\\ \nonumber
& \leq c(N,m,q)E^{-(m-1)/2}\|(X^2_k-X^1_k)(\cdot,\zeta)\|_{L^q(\C)},
\end{align}
since we can find $r,r'$ with $1<r'<2<r<+\infty$ such that $1/r + 1/r' = 1/p$ (note that $p<2$) and thus
\begin{align*}
\left\|\frac{r^{\pm}_1(\cdot,z)}{|\cdot-\lambda|}\right\|_{L^q(\C)}\! \! \! \! &\leq \| r^{\pm}_1(\cdot,z)\|_{L^r(|\eta|<R)}\|1/ |\cdot-\lambda|\|_{L^{r'}(|\eta|<R)}\\
&\quad +\| r^{\pm}_1(\cdot,z)\|_{L^{r'}(|\eta|>R)}\|1/ |\cdot-\lambda|\|_{L^{r}(|\eta|>R)}\\
&\leq c(N,m)E^{-(m-1)/2},
\end{align*}
for $\lambda \in T$ and some fixed $R >1$. Now, since $(X^2_k-X^1_k)(\cdot,\zeta)$ is a continuous $L^q$ solution, $q >2$, of the non-homogeneous $\bar \partial$-equation \eqref{dbarx21} we have, from Lemma \ref{lemtech},
\begin{align}\nonumber
\|(X^2_k-X^1_k)(\cdot,\zeta)\|_{L^q(\C)} \leq c(N) \|(r_2-r_1)(\cdot,z)X^2_k(\cdot,\zeta)\|_{L^{q'}(\C)},
\end{align}
where $1/q' = 1/q + 1/2$. From the fact that $q' < 2$ and $X^2_k$ satisfies \eqref{estx1}, \eqref{estx2}, using H\"older's inequality as above we obtain
\begin{align}\label{estx12q}
\|(X^2_k-X^1_k)(\cdot,\zeta)\|_{L^q(\C)} \leq c(N)\|r_2-r_1\|_{L^{s,s'}(\C)},
\end{align}
for some $1<s<2<s'<+\infty$, therefore
\begin{align} \label{estj2}
J_2(\lambda,\zeta) \leq c(N,m)E^{-(m-1)/2}\|r_2-r_1\|_{L^{s,s'}(\C)},
\end{align}
for $\lambda, \zeta \in T$. Putting together \eqref{estj1}, \eqref{estj2} and \eqref{estj3} we find
\begin{align}
&\| (\Omega^{\pm,2}_k-\Omega^{\pm,1}_k)(\lambda,\cdot)\|_{L^2(T)}\\ \nonumber
&\quad \leq c(N,m)\Bigg(\delta r'_a + \frac{a}{a-1}\left\|\left(\frac{1}{|\lambda|}+|\lambda|\right)(r_2 - r_1)\right\|_{L^{s,s'}(\C \setminus D_a)} \\ \nonumber
&\qquad + E^{-(m-1)/2} \|r_2-r_1\|_{L^{s,s'}(\C)} \Bigg),
\end{align} 
for $\lambda \in T$, since $r_2-r_1 \leq \left(\frac{1}{|\lambda|}+|\lambda|\right)(r_2 - r_1)$.\smallskip

We can finally put everything together and find
\begin{align*}
&\|K^{\pm}_2-K^{\pm}_1\|_{L^2(T)}\leq c(N,m)\Bigg[E^{1/2}\|\rho_2-\rho_1\|_{L^2(T^2)}\\ \nonumber
&\quad +\left\|\left(\frac{1}{|\lambda|}+|\lambda|\right)(r_2 - r_1)\right\|_{L^{s,s'}(\C)}\\ \nonumber
&\quad + E^{-1}\left(\delta r'_a + \frac{a}{a-1}\left\|\left(\frac{1}{|\lambda|}+|\lambda|\right)(r_2 - r_1)\right\|_{L^{s,s'}(\C \setminus D_a)}\right) \Bigg],
\end{align*}
which finishes the proof of Proposition \ref{propKK}.
\end{proof}

\section{Proof of Theorem \ref{maintheo}} \label{secpf}
We start from formula \eqref{formv} and estimate the differences $A_2-A_1$, $B_2 - B_1$ and $C_2-C_1$ separately. We have
\begin{align} \nonumber
A_2-A_1 &= \frac{1}{\pi}\int_{\C}\bigg[-\frac{i\sqrt{E}}{2}\left(\frac{1}{\zeta}+\bar \zeta\right)\left( (r_2-r_1)\overline{e_2}+r_1(\overline{e_2}-\overline{e_1})\right)\\ \nonumber
&\quad+ (r_2-r_1)\overline{\partial_{\bar z}e_2}+r_1\overline{(\partial_{\bar z}e_2 - \partial_{\bar z}e_1)}\bigg]d\Ree \zeta\, d\Imm \zeta.
\end{align}
Using several times H\"older's inequality \eqref{holder}, we find
\begin{align} \label{estaa1}
&|A_2 - A_1| \leq \frac{1}{\pi}\bigg[\frac{\sqrt{E}}{2}\bigg(\left\|\left(\frac{1}{\zeta}+\bar \zeta\right) (r_2-r_1)\right\|_{L^1(\C)}\|e_2(z,\cdot)\|_{L^{\infty}(\C)}\\ \nonumber
&\quad +\left\|\left(\frac{1}{\zeta}+\bar \zeta\right) r_1\right\|_{L^{\tilde p'}(\C)}\|e_2(z,\cdot)-e_2(z,\cdot)\|_{L^{\tilde p}(\C)}\bigg)\\ \nonumber
&\quad +\|r_2-r_1\|_{L^p(\C)}\|\partial_{\bar z}e_2(z,\cdot)\|_{L^{p'}(\C)}\\ \nonumber
&\quad +\|r_1\|_{L^{\tilde p'}(\C)}\|\partial_{\bar z}e_2(z,\cdot)-\partial_{\bar z}e_1(z,\cdot)\|_{L^{\tilde p}(\C)}\bigg],
\end{align}
for $1<p<2$, $\tilde p$ such that $1 / \tilde p = 1/p - 1/2$ and $1/p + 1/p' = 1 / \tilde p + 1 / \tilde p' =1$.

In order to estimate $\|e_2(z,\cdot)-e_2(z,\cdot)\|_{L^{\tilde p}}$ and $\|\partial_{\bar z}e_2(z,\cdot)-\partial_{\bar z}e_1(z,\cdot)\|_{L^{\tilde p}}$ we just remark that from the definition \eqref{defe} $e_j(z,\lambda)$ satisfies
\begin{equation}\nonumber
\frac{\partial}{\partial \bar \lambda}e_j(z,\lambda) = r_j(z,\lambda)\overline{e_j(z,\lambda)}, \quad j=1,2,
\end{equation}
for all $\lambda \in \C$, with $\lim_{\lambda \to \infty}e_j(z,\lambda) = 1$. The operator $\partial_{\bar \lambda}^{-1}$ defined in \eqref{defdbarr}, which intervene in the integral equation defining $e_j(\lambda)$, satisfies the estimate
\begin{equation} \label{estdbarr}
|\partial_{\bar \lambda}^{-1}f(\lambda)|\leq c(p)\|f\|_{L^p_2(\C)}|\lambda|^{2/p-1}, \qquad \text{for } |\lambda| > 1,\quad  p >2,
\end{equation}
which is proved in \cite[Ch. III, (4.16)]{V} (see \eqref{deflpnu} for the definition of $L^p_2(\C)$). As already stated at the beginning of Section \ref{sechil}, since $r_j(\lambda) \in L^p_2(\C)$, equation \eqref{defe} is uniquely solved in $L^q_0(\C)$ ($p/(p-1)\leq q <2$) and in addition $e(\lambda)$ is continuous (see \cite{V}). Then $e(\lambda)$ is $L^{\infty}(\C)$ and since $r_j(\lambda)$ is bounded in $L^p_2(\C)$ for every $p \geq 1$, we obtain
\begin{equation}
|e(\lambda) -1| \leq c(N,p)|\lambda|^{2/p-1}, \qquad |\lambda| >1,
\end{equation}
for every $p >2$, which yields $\|e(\cdot)-1\|_{L^q(\C)} \leq c(N,q)$, for every $q >2$ (since $e$ is continuous on $\C$). The same kind of argument yields 
\begin{align}
\|\partial_{\bar z} e(\cdot)\|_{L^q(\C)} \leq c(N,q), \quad \|\partial_{ z} e(\cdot)\|_{L^q(\C)} \leq c(N,q), \quad \text{for any } q >2.
\end{align}
Thus it is possible to use the same ideas as in \cite[Lemma 4.1]{S2} to estimate $e_j$ as follows:
\begin{align} \label{estmula1}
&\sup_{z \in \C} \|e_2(z,\cdot) - e_1(z,\cdot) \|_{L^{\tilde p}(\C)} \leq c(D,N,p,m) \| r_2 -r_1\|_{L^p(\C)},\\ \label{estmula2}
&\sup_{z \in \C} \left\| \nabla e_2(z,\cdot) - \nabla e_1(z,\cdot)\right\|_{L^{\tilde p}(\C)} \leq c(D,N,p,m)\Bigg[  \| r_2 -r_1\|_{L^p(\C)} \\ \nonumber
&\qquad+ \sqrt{E}\left(\left\| \left(|\lambda| + \frac{1}{|\lambda|}\right)|r_2 - r_1|\right\|_{L^p(\C)} + \| r_2 -r_1\|_{L^p(\C)}  \right) \Bigg],
\end{align}
with $p$ and $\tilde p$ defined above and $\nabla$ is taken with respect to $z$. The proof of \eqref{estmula1} is exactly the same as that of the first estimate of \cite[Lemma 4.1]{S2} while for \eqref{estmula2} the only differences are in some signs, due to \eqref{defr2}, and do not affect the result. 

From Lemma \ref{lemb} we find
\begin{align} \label{estrr1}
\left\|\left(\frac{1}{\zeta}+\bar \zeta\right) r_1\right\|_{L^{\tilde p'}(\C)} &\leq c(N,m,p')E^{-m/2},\\ \label{estrr2}
\|r_1\|_{L^{\tilde p'}(\C)} &\leq c(N,m,p')E^{-m/2}.
\end{align}

Combining estimates \eqref{estmula1}-\eqref{estrr2} with \eqref{estaa1} we find, for a fixed $p \in ]1,2[$,
\begin{align*}
|A_2-A_1|\leq c(D,N,m)\left(\sqrt{E}\left\|\left(\frac{1}{\zeta}+\bar \zeta\right) (r_2-r_1)\right\|_{L^1(\C)}\! \! \! \! \! \!+ \| r_2 -r_1\|_{L^p(\C)}\right).
\end{align*}
Then, using Proposition \ref{propestr} we obtain
\begin{align}\label{esta12}
|A_2-A_1|&\leq c(D,N,m)\bigg[ E^{-1/2}\left(E^{1/2}+\kappa \log(3+\delta^{-1})\right)^{-(m-2)} \\ \nonumber
&\quad + \delta(3+\delta^{-1})^{2\kappa(l+1)}\bigg],
\end{align}
for $\kappa$ and $\delta$ as in the statement. We now pass to $B_2-B_1$, which is given by
\begin{align} \nonumber
B_2-B_1 &= \frac{1}{2\pi i}\int_{|\zeta|=1}\bigg\{\left(\partial_z K_2(\zeta)-\partial_z K_1(\zeta)\right)\\ \nonumber
&\quad \times \left[-1+\frac{1}{\pi}\int_{\C}r_2(\lambda,z)\overline{\Omega^2_2(\lambda,\zeta)}d\Ree \lambda \, d\Imm \lambda \right]\\ \nonumber
&\quad + \partial_z K_1(\zeta)\bigg[\frac{1}{\pi}\int_{\C}(r_2-r_1)(\lambda,z)\overline{\Omega_2^2(\lambda,\zeta)}d\Ree \lambda\, d\Imm \lambda \\ \nonumber
&\quad + \frac{1}{\pi}\int_{\C}r_1(\lambda,z)\left(\overline{\Omega_2^2(\lambda,\zeta)}-\overline{\Omega_2^1(\lambda,\zeta)}\right)d\Ree \lambda\, d\Imm \lambda\bigg]\\ \nonumber
&\quad +\frac{(K_2-K_1)(\zeta)}{\pi}\int_{\C}r_2(\lambda,z)\bigg[-\frac{i\sqrt{E}}{2}\left(\frac{1}{\lambda}+\bar \lambda\right)\overline{\Omega_2^2(\lambda,\zeta)}\\ \nonumber
&\quad +\overline{\partial_{\bar z}\Omega_2^2(\lambda,\zeta)}\bigg]d\Ree \lambda\,d\Imm \lambda + \frac{K_1(\zeta)}{2\pi}\int_{\C}(r_2-r_1)(\lambda,z)\\ \nonumber
&\quad \times\left[-\frac{i\sqrt{E}}{2}\left(\bar \lambda + \frac{1}{\lambda}\right)\overline{\Omega_2^2(\lambda,\zeta)}+\overline{\partial_{\bar z}\Omega_2^2(\lambda,\zeta)}\right]\\ \nonumber
&\quad +r_1(\lambda,z)\bigg[-\frac{i\sqrt{E}}{2}\left(\bar \lambda + \frac{1}{\lambda}\right)\left(\overline{\Omega_2^2(\lambda,\zeta)}-\overline{\Omega_2^1(\lambda,\zeta)}\right)\\ \nonumber 
&\quad +\overline{\partial_{\bar z}\Omega_2^2(\lambda,\zeta)}-\overline{\partial_{\bar z}\Omega_2^1(\lambda,\zeta)} \bigg] d\Ree \lambda \, d\Imm \lambda\bigg\}d\zeta.
\end{align}
This yields
\begin{align*}
|B_2-B_1| &\leq \frac{1}{2\pi}\left\|\partial_z K_2-\partial_z K_1\right\|_{L^2(T)} \\ &\quad \times \left\|-1+\frac{1}{\pi}\int_{\C}r_2(\lambda,z)\overline{\Omega^2_2(\lambda,\cdot)}d\Ree \lambda \, d\Imm \lambda \right\|_{L^2(T)}\\ \nonumber
&\quad + \|\partial_z K_1\|_{L^2(T)}\bigg\|\frac{1}{\pi}\int_{\C}(r_2-r_1)(\lambda,z)\overline{\Omega_2^2(\lambda,\zeta)}d\Ree \lambda\, d\Imm \lambda \\ \nonumber
&\quad + \frac{1}{\pi}\int_{\C}r_1(\lambda,z)\left(\overline{\Omega_2^2(\lambda,\cdot)}-\overline{\Omega_2^1(\lambda,\zeta)}\right)d\Ree \lambda\, d\Imm \lambda\bigg\|_{L^2(T)}\\ \nonumber
&\quad +\frac{\|K_2-K_1\|_{L^2(T)}}{\pi}\bigg\|\int_{\C}r_2(\lambda,z)\bigg[-\frac{i\sqrt{E}}{2}\left(\frac{1}{\lambda}+\bar \lambda\right)\overline{\Omega_2^2(\lambda,\cdot)}\\ \nonumber
&\quad +\overline{\partial_{\bar z}\Omega_2^2(\lambda,\cdot)}\bigg]d\Ree \lambda\,d\Imm \lambda \bigg\|_{L^2(T)}+ \frac{\|K_1\|_{L^2(T)}}{2\pi}\bigg\|\int_{\C}(r_2-r_1)(\lambda,z)\\ \nonumber
&\quad \times\left[-\frac{i\sqrt{E}}{2}\left(\bar \lambda + \frac{1}{\lambda}\right)\overline{\Omega_2^2(\lambda,\cdot)}+\overline{\partial_{\bar z}\Omega_2^2(\lambda,\cdot)}\right]\\ \nonumber
&\quad +r_1(\lambda,z)\bigg[-\frac{i\sqrt{E}}{2}\left(\bar \lambda + \frac{1}{\lambda}\right)\left(\overline{\Omega_2^2(\lambda,\cdot)}-\overline{\Omega_2^1(\lambda,\cdot)}\right)\\ \nonumber 
&\quad +\overline{\partial_{\bar z}\Omega_2^2(\lambda,\cdot)}-\overline{\partial_{\bar z}\Omega_2^1(\lambda,\cdot)} \bigg] d\Ree \lambda \, d\Imm \lambda \bigg\|_{L^2(T)}.
\end{align*}
Using H\"older's inequality, Lemma \ref{lemestrr}, estimates \eqref{estomega2}, \eqref{estkpm}, \eqref{estx12q}, \eqref{estomega+}, \eqref{estomega-}, \eqref{estK1} we find
\begin{align*}
|B_2-B_1| &\leq c(N,m)\Bigg[\left\|\partial_z K_2-\partial_z K_1\right\|_{L^2(T)} + \left\|\left(\frac{1}{|\lambda|}+|\lambda|\right)(r_2 - r_1)\right\|_{L^{s,s'}(\C)}\\
&\quad + \frac{\|K_2-K_1\|_{L^2(T)}}{\sqrt{E}^{m-1}}+ \frac{\sup_{\zeta \in T}\|\partial_{\bar z}\Omega_2^2(\cdot,\zeta)-\partial_{\bar z}\Omega_2^1(\cdot,\zeta)\|_{L^q(\C)}}{\sqrt{E}^{m-1}}\Bigg],
\end{align*}
for some $1<s<2<s'<+\infty$, $q > 2$. We estimate the last term using Lemma \ref{lemtech}. Since $(X^{\pm,2}_k-X^{\pm,1}_k)(\cdot,\zeta)$, defined in \eqref{defxpm}, is a continuous $L^q$ solution, $q >2$, of the non-homogeneous $\bar \partial$-equation \eqref{dbarpm} we have, from Lemma \ref{lemtech},
\begin{align*}
\|(X^{\pm,2}_k-X^{\pm,1}_k)(\cdot,\zeta)\|_{L^q(\C)} &\leq c(N)\Big( \|(r^{\pm}_2-r^{\pm}_1)(\cdot,z)X^2_k(\cdot,\zeta)\|_{L^{q'}(\C)}\\
&\quad + \|r^{\pm}_1(\cdot,z)(X^2_k-X^1_k)(\cdot,\zeta)\|_{L^{q'}(\C)} \\
&\quad + \|(r_2-r_1)(\cdot,z)X^{\pm,2}_k(\cdot,\zeta)\|_{L^{q'}(\C)}\Big) \\
&\leq c(N,m)\bigg(\sqrt{E}\left\|\left(\frac{1}{|\lambda|}+|\lambda|\right)(r_2 - r_1)\right\|_{L^{s,s'}(\C)}\\
&\quad +\sqrt{E}^{-(m-1)}\|(X^2_k-X^1_k)(\cdot,\zeta)\|_{L^{r}(\C)}\\
&\quad + \sqrt{E}^{-(m-1)}\|(r_2-r_1)(\cdot,z)\|_{L^{s,s'}(\C)}\bigg)
\end{align*}
where $1/q' = 1/q + 1/2$ and $r >2$. Here we used several times H\"older's inequality, the fact that $q' < 2$ and that $X^j_k$, $X^{\pm,j}_k$ satisfy \eqref{estx1}, \eqref{estx2} and \eqref{estx+}, \eqref{estx-}. From \eqref{estx12q} and the fact that $r_2-r_1 \leq \left(\frac{1}{|\lambda|}+|\lambda|\right)(r_2 - r_1)$ we obtain
\begin{align*}
\|(X^{\pm,2}_k-X^{\pm,1}_k)(\cdot,\zeta)\|_{L^q(\C)} \leq c(N,m)\sqrt{E}\left\|\left(\frac{1}{|\lambda|}+|\lambda|\right)(r_2 - r_1)\right\|_{L^{s,s'}(\C)},
\end{align*}
 which yields
\begin{align}\label{estomega12qq}
\|(\nabla \Omega^{2}_k-\nabla \Omega^{1}_k)(\cdot,\zeta)\|_{L^q(\C)} \leq c(N,m)\sqrt{E}\left\|\left(\frac{1}{|\lambda|}+|\lambda|\right)(r_2 - r_1)\right\|_{L^{s,s'}(\C)},
\end{align}
for some $1<s<2<s'<+\infty$, $q >2$.

Now, Propositions \ref{propK} and \ref{propKK} as well as estimate \eqref{estomega12qq} gives
\begin{align*}
|B_2-B_1| &\leq c(N,m)\Bigg[E^{1/2}\|\rho_2-\rho_1\|_{L^2(T^2)}\\ \nonumber
&\quad +\left\|\left(\frac{1}{|\lambda|}+|\lambda|\right)(r_2 - r_1)\right\|_{L^{s,s'}(\C)}\\ \nonumber
&\quad + E^{-1}\left(\delta r'_a + \frac{a}{a-1}\left\|\left(\frac{1}{|\lambda|}+|\lambda|\right)(r_2 - r_1)\right\|_{L^{s,s'}(\C \setminus D_a)}\right) \Bigg].
\end{align*}
From Lemma \ref{lemdifh} and \eqref{estexp} we find
\begin{align}
\delta r'_a \leq c(D,N,m) e^{4(l+1)\sqrt{E}(a-1)}\delta.
\end{align}
Like in the proof of Proposition \ref{propestr} we define
\begin{equation} \label{defa}
 a = 1 + \frac{\kappa \log (3+\delta^{-1})}{\sqrt{E}},
\end{equation}
for $\kappa < 1/4(l+1)$. Note that
\begin{equation}
\frac{a}{a-1} = 1+ \frac{\sqrt{E}}{\kappa \log (3+\delta^{-1})}\leq 1+\sqrt{E},
\end{equation}
for $\delta < \frac{1}{e^{1/\kappa}-3}$. Repeating the proof of Proposition \ref{propestr} we obtain
\begin{align}
&\delta r'_a + \frac{a}{a-1}\left\|\left(\frac{1}{|\lambda|}+|\lambda|\right)(r_2 - r_1)\right\|_{L^{s,s'}(\C \setminus D_a)} \\ \nonumber
&\quad \leq c(D,N,m)\left( \delta(3+\delta^{-1})^{4\kappa(l+1)} + E^{-\frac 1 2}\left(E^{\frac 1 2}+\kappa \log(3+\delta^{-1})\right)^{-(m-2)}\right),
\end{align}
for $\delta < \frac{1}{e^{1/\kappa}-3}$, $\kappa < 1/4(l+1)$. Then, using Propositions \ref{proprho} and \ref{proprho} we get
\begin{align} \label{estb12}
 |B_2-B_1| &\leq c(D,N,m)\bigg( \sqrt{E}\delta(3+\delta^{-1})^{4\kappa(l+1)}\\ \nonumber 
&\quad + {E}^{-1}\left(E^{\frac 1 2}+\kappa \log(3+\delta^{-1})\right)^{-(m-2)}\bigg),
\end{align}
for $\delta < \frac{1}{e^{1/\kappa}-3}$, $\kappa < 1/4(l+1)$, $E >E_4$.

We need now to estimate $C_2-C_1$, which can be written as follows:
\begin{align} \nonumber
C_2-C_1 &= -\frac{1}{2\pi i}\int_{|\zeta|=1}\bigg\{(\overline{\partial_{\bar z} K_2-\partial_{\bar z} K_1})(\zeta)\\ \nonumber
&\quad \times \left[\frac{1}{\pi}\int_{\C}r_2(\lambda,z)\overline{\Omega^2_1(\lambda,\zeta)}d\Ree \lambda \, d\Imm \lambda \right]\\ \nonumber
&\quad + \overline{\partial_{\bar z} K_1}(\zeta)\bigg[\frac{1}{\pi}\int_{\C}(r_2-r_1)(\lambda,z)\overline{\Omega_1^2(\lambda,\zeta)}d\Ree \lambda\, d\Imm \lambda \\ \nonumber
&\quad + \frac{1}{\pi}\int_{\C}r_1(\lambda,z)\left(\overline{\Omega_1^2(\lambda,\zeta)}-\overline{\Omega_1^1(\lambda,\zeta)}\right)d\Ree \lambda\, d\Imm \lambda\bigg]\\ \nonumber
&\quad +\frac{(\overline{K_2-K_1})(\zeta)}{\pi}\int_{\C}r_2(\lambda,z)\bigg[-\frac{i\sqrt{E}}{2}\left(\frac{1}{\lambda}+\bar \lambda\right)\overline{\Omega_1^2(\lambda,\zeta)}\\ \nonumber
&\quad +\overline{\partial_{\bar z}\Omega_1^2(\lambda,\zeta)}\bigg]d\Ree \lambda\,d\Imm \lambda + \frac{\overline{K_1}(\zeta)}{2\pi}\int_{\C}(r_2-r_1)(\lambda,z)\\ \nonumber
&\quad \times\left[-\frac{i\sqrt{E}}{2}\left(\bar \lambda + \frac{1}{\lambda}\right)\overline{\Omega_1^2(\lambda,\zeta)}+\overline{\partial_{\bar z}\Omega_1^2(\lambda,\zeta)}\right]\\ \nonumber
&\quad +r_1(\lambda,z)\bigg[-\frac{i\sqrt{E}}{2}\left(\bar \lambda + \frac{1}{\lambda}\right)\left(\overline{\Omega_1^2(\lambda,\zeta)}-\overline{\Omega_1^1(\lambda,\zeta)}\right)\\ \nonumber 
&\quad +\overline{\partial_{\bar z}\Omega_1^2(\lambda,\zeta)}-\overline{\partial_{\bar z}\Omega_1^1(\lambda,\zeta)} \bigg] d\Ree \lambda \, d\Imm \lambda\bigg\}d\zeta.
\end{align}
We proceed exactly as for $B_2-B_1$ and we find
\begin{align*}
|C_2-C_1|&\leq c(N,m)\Bigg[\left\|\left(\frac{1}{|\lambda|}+|\lambda|\right)(r_2 - r_1)\right\|_{L^{s,s'}(\C)} \\ \nonumber
&\quad + \frac{\|K_2-K_1\|_{L^2(T)}}{\sqrt{E}^{m-1}}+ \frac{\left\|\partial_{\bar z} K_2-\partial_{\bar z} K_1\right\|_{L^2(T)}}{\sqrt{E}^{m}} \Bigg],
\end{align*}
for some $1<s<2<s'<+\infty$. Here we used again H\"older's inequality as well as Lemma \ref{lemestrr}, estimates \eqref{estomega1}, \eqref{estkpm}, \eqref{estx12q}, \eqref{estomega+}, \eqref{estomega-} and \eqref{estomega12qq}. Using Propositions \ref{propK} and \ref{propKK} with $a$ defined in \eqref{defa} and arguing as for $B_2-B_1$ we obtain, with Propositions \ref{proprho} and \ref{propestr},
\begin{align}\label{estc12}
|C_2-C_1|&\leq c(D,N,m) \bigg( \delta(3+\delta^{-1})^{4\kappa(l+1)}\\ \nonumber 
&\quad + {E}^{-1}\left(E^{\frac 1 2}+\kappa \log(3+\delta^{-1})\right)^{-(m-2)}\bigg),
\end{align}
for $\delta < \frac{1}{e^{1/\kappa}-3}$, $\kappa < 1/4(l+1)$, $E >E_4$.

We can now put estimates \eqref{esta12}, \eqref{estb12} and \eqref{estc12} together and from \eqref{formv} find
\begin{align}
\|v_2-v_1\|_{L^{\infty}(D)}&\leq c(D,N,m)\bigg( {E}\delta(3+\delta^{-1})^{4\kappa(l+1)}\\ \nonumber 
&\quad + \left(E^{\frac 1 2}+\kappa \log(3+\delta^{-1})\right)^{-(m-2)}\bigg),
\end{align}
for $\delta < \frac{1}{e^{1/\kappa}-3}$, $\kappa < 1/4(l+1)$, $E >E_1 = \max (E_0,E_2,E_3,E_4)$. Now, for every $0 < \tau < 1$ there is a $0 <\kappa< 1/4(l+1)$ such that $\tau = 1 - 4\kappa(l+1)$. Then we have
\begin{align}
\|v_2-v_1\|_{L^{\infty}(D)}&\leq c(D,N,m)\left( {E}\delta^{\tau}+ \left(E^{\frac 1 2}+(1-\tau) \log(3+\delta^{-1})\right)^{-(m-2)}\right),
\end{align}
for $\delta < \delta_{\tau}$ and $E >E_1$. This finishes the proof of Theorem \ref{maintheo}. \qed

\end{document}